\documentclass[11pt]{amsart}
\input amssymb.sty
\usepackage{amsmath}
\usepackage{amsfonts}
\usepackage{epsfig}
\usepackage{enumerate}

\newcommand{\CP}{\mathbb{CP}}
\newcommand{\pp}{\mathbb{CP}}
\newcommand{\C}{\mathbb{C}}
\newcommand{\Z}{\mathbb{Z}}

\newcommand{\G}{\Gamma}
\newcommand{\g}{\gamma}
\newcommand{\T}{\cdot}

\newcommand{\p}{\pi_1}

\newcommand{\vp}{\varphi}
\newcommand{\ri}{\rightarrow}

\newcommand{\Sing}{ \mathop{\rm Sing} }
\newcommand{\Int}{ \mathop{\rm Int} }
\newcommand{\Graph}{ \mathop{\rm Graph} }

\newcommand{\tB}{\widetilde{B}}

\newcommand{\un}{\underline}

\def\lm{\lambda}
\def\Dl{\Delta}
\def\dl{\delta}

\newcommand{\usr}{\underset\sim\rightarrow}

\newcommand{\uz}{\overline{Z}}
\newcommand{\uzs}{\overline{Z}^2}
\newcommand{\lz}{\underline{Z}}

\begin{document}

\newtheorem{thm}{Theorem} [section]

\newtheorem{corollary}[thm]{Corollary}
\newtheorem{lemma}[thm]{Lemma}
\newtheorem{prs}[thm]{Proposition}
\newtheorem{defi}[thm]{Definition}
\newtheorem{remark}[thm]{Remark}
\newtheorem{conjecture}[thm]{Conjecture}
\newtheorem{example}[thm]{Example}
\newtheorem{notation}[thm]{Notation}
\newtheorem{condition}[thm]{Condition}
\newtheorem{construction}[thm]{Construction}
\newtheorem*{condition*}{Condition}

\title[On fundamental groups related to degeneratable surfaces]{On fundamental groups related to degeneratable surfaces: conjectures and examples}

\author[M. Friedman, M. Teicher]{Michael Friedman and Mina Teicher$^1$}\address{
Michael Friedman, Mina Teicher, Department of Mathematics,
 Bar-Ilan University, 52900 Ramat Gan, Israel}
 \email{fridmam@macs.biu.ac.il, teicher@macs.biu.ac.il}
\stepcounter{footnote} \footnotetext{This work is partially
supported by the Emmy Noether Research Institute for Mathematics
(center of the Minerva Foundation of Germany).}

\maketitle
\begin{abstract}

We argue that for a smooth surface $S$, considered as a ramified cover over $\pp^2$, branched over a
nodal-cuspidal curve $B \subset \CP^2$, one could use
the structure of the fundamental group of the complement of the branch curve
$\pi_1(\pp^2 - B)$ to understand other properties of the surface and its degeneration and vice-versa.
In this paper, we look at embedded-degeneratable surfaces - a class of surfaces admitting a planar
 degeneration with a few combinatorial conditions imposed on its degeneration.  We close a
  conjecture of Teicher on the virtual solvability of  $\pi_1(\pp^2 - B)$ for these surfaces and present
  two new conjectures on the structure of this group, regarding non-embedded-degeneratable surfaces.
  We  prove two theorems supporting our conjectures, and show that for  $\pp^1 \times C_g$,
  where $C_g$ is a curve of genus $g$, $\pi_1(\pp^2 - B)$ is a quotient of an Artin group associated to the degeneration.

\end{abstract}

\tableofcontents

\section[Introduction]{Introduction}

Given a smooth algebraic projective variety $X$, one of the main techniques used to obtain information
on $X$ is to degenerate it to a union of ``simpler" varieties. The ``simplest" degeneration can be thought
as the degeneration of $X$ to a union of $\dim X$--planes, and one would like to use the combinatorial data induced
from this arrangement of planes in order to find (or bound) certain invariants of $X$.

When $\dim X = 1$, one would like to degenerate the curve into a
line arrangement with only nodes as the singularities. This has been thoroughly
investigated. For example, it is known that any smooth
plane curve can be degenerated into a union of lines. However, the
situation for a curve in $\pp^n,\,n > 2$ is completely different
as there are, for example, smooth curves in $\pp^3$ which cannot
be degenerated into a line arrangement with only double points
(see \cite{Ha}).

When $\dim X = 2$, the problem of investigating projective surfaces
in terms of their degeneration to a union of planes has only been
investigated partially (see, for example, Zappa's papers from the
1940's \cite{Z} and \cite{CCFR1} for a survey on this topic; see
also \cite{Ku3} for degeneration of surfaces in $\pp^3$). Also,
one should allow the existence of more complicated singularities in
order to obtain degenerations. But there is another method to
extract information on the surface, which is to consider it as a
branched cover of the projective plane $\pp^2$ with respect to a
generic projection. The motivation for this point of view is
Chisini's conjecture (recently proved by Kulikov
\cite{Ku},\cite{Ku2}): Let $B$ be the branch curve of generic
 projection $\pi: S \to \pp^2$ of degree at
least 5. Then $(S,\pi)$ is uniquely determined by the pair
$(\pp^2,B)$. Moreover, if two surfaces $S_1$ and $S_2$ are
deformation equivalent, then their branch curves $B_1$ and $B_2$
are isotopic. Thus, if the fundamental group $\pi_1(\C^2 - B_1)$
is not isomorphic to $\pi_1(\C^2 - B_2)$ then the surfaces are not
deformation equivalent. This gives another motivation for
considering $S$ in terms of its branch curve.

Therefore, it is reasonable to combine the two methods outlined above, i.e., investigating a projective surface
$S$ and its degeneration $S_0$ by looking at their branch curves $B$ and $B_0$. Explicitly, we want to find the relations between the combinatorics of the planar degeneration and the fundamental group $\pi_1(\C^2 - B)$.

Several works were done in this direction: for different
embeddings of $\pp^1 \times \pp^1$, for the Veronese surface $V_n$
(\cite{MoTe4}, \cite{MoTe5} for  $V_n, n\geq 3$ and \cite{Za} for
$V_2$), for the Hirzebruch surfaces $F_{1,(a,b)}$ and
$F_{2,(2,2)}$ (\cite{Fr2}, \cite{AFT2}), for $K3$ surfaces
(\cite{FrTe}), for a few toric surfaces and for $\pp^1 \times
\mathbb{T}$ (where $ \mathbb{T}$ is a complex torus. see
\cite{AFT}). For each surface in this list one can associate a
graph $T$ to the degenerated surface. In all of the examples
mentioned above the fundamental group $\pi_1(\C^2 - B)$ is either
a quotient of an associated Artin group $A(T)$ (except the
Veronese surface $V_2 \subset \pp^5$) or a quotient of a subgroup
of $\widetilde{A}(T) \times \widetilde{A}(T)$ (where
$\widetilde{A}(T)$ is a quotient of $A(T)$ by a single relation.
For example, when $T$ is a tree with maximum valence $3$, then
$A(T)$ is isomorphic to the braid group $B_n$, where $n = $ degree
of the surface). In particular, once the embedding of the surface
in a projective space is ``ample enough", the structure of
$\pi_1(\C^2 - B)$ is of the mentioned second type. Thus, a natural
question rises: what are the sufficient conditions on the
degeneration such that $\pi_1(\C^2 - B)$ will be isomorphic to a  quotient of a
subgroup of $\widetilde{A}(T) \times \widetilde{A}(T)$? One of the
goals of this paper is to give the conditions under which the
fundamental group has this desired structure. These conditions are in a form of a local--global condition: if
there are enough singular points in the degenerated surface satisfying a certain local condition, then the
fundamental group is isomorphic to the quotient. Under these
conditions, the conjecture posed in \cite{Te1} regarding the
virtual-solvability of the above fundamental group is proven.
%

Another main result deals with a new set of examples, not
satisfying these conditions. The surfaces $\pp^1 \times C_g$,
where $C_g$ is a curve of genus $g \geq 1$, are studied, and for
$g \geq 1$ the above fundamental group is computed. These new examples
are essential for the second goal of this paper: to understand
better these groups for non-simply connected surfaces.

The structure of the paper is as follows. Section \ref{sec2} examines the structure of the fundamental group. Subsections \ref{subsec2cond} and \ref{subsecCondDeg} introduce  the main definitions and restrictions on the
degeneration. We then state one of the main theorems in Subsection \ref{subsec2MainThm}: that under a certain condition, there is an epimorphism from $\widetilde{B}^{(2)}_n$ to  $\pi_1(\C^2 - B)$.
We also present two conjectures on the structure of $\pi_1(\C^2 - B)$ when the condition does not hold (see Conjectures \ref{conjStructure} and \ref{conjStructure2}). In Subsection \ref{subsecProof} we prove the main theorem from Subsection \ref{subsec2MainThm}. In Section \ref{sec3} we prove another main theorem, where we compute the  group $G_g = \pi_1(\C^2 - B_g)$, where $B_g$ is the branch curve of $\pp^1 \times C_g$. We show that $G_g$ is (again) a quotient of $A(T)$, and also compute $\pi_1((\pp^1 \times C_g)_{Gal})$ -- the fundamental group of the Galois cover  of $\pp^1 \times C_g$.

 \textbf{Acknowledgements}:
 We thank Alberto Calabri and Ciro Ciliberto for refereing the first author to their paper  \cite{CCFR} and for fruitful discussions  during the ``School (and Workshop) on the Geometry of Special Varieties" which was
 held at 2007 at the IRST,  Fondazione Bruno Kessler in Povo (Trento).
 We also would like to thank Christian Liedtke and Robert Schwartz for stimulating talks and important discussions.

\section[Degenerations and fundamental groups]{Degenerations and fundamental groups} \label{sec2}
In this section we examine the structure of the fundamental group
of the complement of the branch curve, under some assumptions.
Subsection \ref{subsec2cond} introduces the main definitions and
notations. We state the main theorem on the structure of the
fundamental group of the complement of the branch curve in $\C^2$,
under certain conditions, in Subsection \ref{subsec2MainThm} and
also present two conjectures on the structure of the fundamental
group regarding surfaces which do not satisfy the desired
conditions. The virtual--solvability of this fundamental group is
discussed in Subsection \ref{subsecVirt}, together with the class
of surfaces satisfying the desired conditions. In Subsection
\ref{subsecProof} we prove the main theorem.

\subsection[Notations for planar degeneration]{Notations for planar degeneration} \label{subsec2cond}
We begin with a few definitions.

\begin{defi} \label{defDeg}\begin{enumerate} \item [\emph{(i)}] \emph{Degeneration: Let $\Delta$ be the complex unit disc. A degeneration of surfaces, parametrized by $\Delta$ is a proper and flat morphism $\rho : S \rightarrow \Delta $ (where $S$ is a 3--dim variety) such that each fibre $S_t = \rho^{-1}(t)$, $t \neq 0$ (where 0 is the
closed point of $\Delta$), is a smooth, irreducible, projective
surface. The fiber $S_0$ is called the central fiber. A
degeneration $\rho : S \rightarrow \Delta $ is said to be embedded
in $\pp^r$ if there is an inclusion $i:S \hookrightarrow \Delta
\times \pp^r$ and, when denoting the projection $p_1 : \Delta
\times \pp^r \rightarrow \Delta$, then $p_1i = \rho$. } \item
[\emph{(ii)}] \emph{Planar degeneration: When the central fiber
$S_0$ in the above embedded degeneration is a union of planes,
then we call the degeneration a planar degeneration. A survey on degeneratable  surfaces can
be found in \cite{CCFR1}. Examples of
planar degenerations can be found in
\cite{CCFR} (for scrolls), \cite{MoRoTe} (for Hirzebruch surfaces), \cite{MoTe3} (for veronese surfaces), \cite{Mo2} (for $\pp^1 \times \pp^1$).} \item [\emph{(iii)}]
\emph{Regeneration: The regeneration methods are actually,
locally, the reverse process of the degeneration method. In this
article it is used as a generic name for finding a degeneration
$\rho : S \rightarrow \Delta $ when the central fiber $S_0$ is
given. In fact, one can deduce what is the effect of a
regeneration on the corresponding branch curves. The regeneration
rules (see Subsection \ref{subsecRegRule}) explain the effect of
the regeneration on the braid monodromy factorization (see
Subsection \ref{subsecBMF}) of the branch curves of the fibers. }
\item [\emph{(iv)}] \emph{Local fundamental group: Given a planar
degeneration $\rho : S \rightarrow \Delta $, denote by $B_t$ the
branch curve of a generic projection of $S_t$ to $\pp^2$ (such that the center
of projection is the same for every $t$). Given a singular point
$p \in B_0$ choose a small neighborhood $U$ of $p$ such that $U
\cap \Sing(B_0) = \{p\}$. Since $S_0$ is a planar degeneration,
there are lines $\ell_i$ such that  $U\cap B_0 = \cup (U \cap \ell_i) $, such that $\cap
\ell_i = \{p\}$. Assume that for the branch curve $B_t$ of general
fiber $S_t$,$t \neq 0$, we have that $\lim_{t \rightarrow 0}(U
\cap \Sing(B_t)) = \{p\}$. The local fundamental group of $p$ is
defined as $\pi_1(U - B_t)$ and we denote it by $G_p$.}
\end{enumerate}
\end{defi}

Let $S_1 \subset \CP^N$ be a  smooth surface of degree $n$ which admits a planar degeneration
$\rho : S \rightarrow \Delta$, and let $f : \pp^N
\to \pp^2$ be the generic projection w.r.t. $S_t$, for every $t$. We
denote by $R$ the ramification curve of $S_1$ and by $B \subset
\pp^2$ its branch curve with respect to a generic projection $\pi
\doteq  f|_{S_1}: S_1 \ri \CP^2$. Also, let $G \doteq \p(\CP^2 - B)$.

We denote by $S_0$ the planar degeneration of $S_1$ (the central fiber of $\rho$), i.e. $S_0$ is a
union of planes. Let
$$
 S_0 = \bigcup_{i=1}^n \Pi_i
$$
 such that each $\Pi_i$ is a plane.
\begin{notation}
\emph{$n = \deg S_0 = \deg S_1$.}
\end{notation}

Let $\pi_0 \doteq f|_{S_0}  : S_0 \to \pp^2$ be the generic projection of
$S_0$ to $\pp^2$. In this case, the ramification curve also
degenerates into a union of $\ell$ lines
$$
R_0 = \bigcup_{i=1}^{\ell} L_i,
$$
and thus the degenerated branch curve is
of the form $$\pi_0(R_0) = B_0 = \bigcup_{i=1}^{\ell} l_i,$$ where
$\l_i = \pi_0(L_i)$.
\begin{notation}
\emph{$\ell = \deg R_0 = \deg B_0$.} 
\end{notation}

Since $R_0$ is an  arrangement of lines in $\pp^N$, these lines
can intersect each other.

\begin{notation}
\emph{$m' = $ the number of points $\{P_i\}_{i=1}^{m'}$
which lie on more than one line $L_j$.}
\end{notation}
 For a point $x \in L_i$ (or
$x \in l_i$) let $v(x)$ be the number of distinct lines in $R_0$ (or $B_0$) on which $x$
lies. For example, if $x \in \{P_i\}_{i=1}^{m'}$, then $v(x) > 1$.

\begin{notation}
\emph{ Denote by $P = \{x \in B_0 : v(x) > 1\}$, and let $p_i =
\pi_0(P_i) \in B_0$ (note that $v(p_i) = v(P_i)$). Denote $P'
\doteq \{p_i\}_{i=1}^{m'}$.}
\end{notation}

\begin{remark}
\emph{ Note also that $P' = \{p_i\}_{i=1}^{m'} \varsubsetneq P$,
since there are points (called {\em parasitic intersection}
points; see the explanation in Subsection \ref{subsecParasitic})
which are in $P$ but not in $\{p_i\}_{i=1}^{m'}$.}
\end{remark}

\begin{notation} \label{defM}
\emph{Recall that $R_0 = \cup L_i$. Define the set of lines
$$
M \doteq \{L_i \in R_0 :\mbox{ there is only one point } x\in L_i
\mbox{ such that } v(x) > 1\}.
$$
For each $l \in M$, choose a point $y_l \in l$ s.t. $v(y_l)=1$.
Denote $$Y \doteq \{y_l\}_{l \in M}; $$   the set of points is
called the set of 2--point.}
\end{notation}

We recall the definition of  $\widetilde{B}_n$, since the local
fundamental group of many of the singular points of $B_0$ is
strongly related to this group.

\begin{defi} \label{defBn} \emph{(1) Let $X,Y$ be two half-twists in the braid group $B_n = B_n(D,K)$
(see Subsection \ref{subsecBMF} for the notation $D,K$). We say
that $X,Y$ are \emph{transversal} if they are defined by two
simple paths $\xi, \eta$ which intersect
transversally in one point  different from their ends.\\
(2) Let $N$ be the normal subgroup of $B_n$ generated by conjugates
of $[X,Y]$, where $X,Y$ is a transversal pair of half-twists.
Define
$$\widetilde{B}_n = B_n/N.$$ Let $x_1,...,x_{n-1}$ be the standard generators of $B_n$.
Equivalently, we can define $$\tB_n = B_n/\langle\!\langle
[x_2,x_3^{-1}x_1^{-1}x_2x_1x_3]\rangle\!\rangle$$ for $n > 3$.
 Recall
that we can define on $B_n$ two natural
homomorphism: \\(i) $\deg\,:\,B_n \to \Z$ s.t. $\deg(\prod x_i^{n_i}) = \sum n_i$.\\
(ii) $\sigma\,:\,B_n \to S_n$ s.t. $\sigma(x_i) = (i\,\,i+1)$.
For properties of $\widetilde{B}_n$ see, for example, \cite{Mo},\cite{RobbT}, \cite{Te2}.\\\\
(3) The following group plays an important role in finding a
presentation of a fundamental group of the complement of the
branch curve. Define, as in \cite{denis}, the group
$$
\tB_n^{(2)} \doteq \{(x,y) \in \tB_n \times \tB_n, deg(x) =
deg(y), \sigma(x) = \sigma(y)\}. $$}
\end{defi}

\begin{defi} \label{defQ} \emph{
Recall that for $p \in P' = \{p_i\}_{i=1}^{m'}$, we denote by $G_p$ the local fundamental group (see Definition \ref{defDeg}(iv)). Define the following set: \begin{center} $Q \doteq \{ p \in \{p_i\}_{i=1}^{m'} :$ there
exists an epimorphism of $ \tB^{(2)}_{v(p)} \twoheadrightarrow
G_p,\,$ and $ v(p) > 3\}$ \end{center} and denote $$|Q| = m.$$}
\end{defi}

Thus, we have the following relations between the sets of points:
 $$Q \doteq \{x_j\}_{j = 1}^m \subset P' \doteq \{p_i\}_{i=1}^{m'}
\subset P.$$

\begin{remark}\emph{
The definition of $Q$ is not meaningless: there are singular
points $p \in \{p_i\}_{i=1}^{m'},$\\$\, v(p)~>~3$ which occur during
the (described above) degeneration process have an epimorphism
$\tB^{(2)}_{v(p)} \twoheadrightarrow G_p$, where $G_p$ is the
local fundamental group associated to $p$. For example, let $p_6$
(resp. $p_5$) a singular point of $S_0$ called a $6-$point (resp.
$5-$point) which is locally an intersection of 6 (resp. 5) planes
at a point, whose regeneration is described at \cite{Mo}
\cite[Definition 4.3.3]{MoRoTe} (resp. \cite{Fr1}). Then $G_{p_i}
$ is isomorphic to a quotient of $\tB^{(2)}_i$ for $i=6,5$. For
the $4-$point $p_4$ (s.t. its regeneration is described at
\cite{Mo}, \cite{RobbT}), we get that $G_{p_4} \simeq \tB_4$, which
is also a quotient of $\tB^{(2)}_4$. }
\end{remark}

\begin{defi}[$\Graph_{S_0}$] \label{defGraphS0}
\emph{We define the graph $\Graph_{S_0}$.
 The vertices are the $m'$ points $\{P_i\}_{i=1}^{m'}$ and the set $Y$ of 2--points. Two vertices in $\Graph_{S_0}$ are connected by an edge if both of the
corresponding points on $R_0$ lie on a unique line $L_i \subset R_0$.}
\end{defi}

 We want to defined boundary and interior (non--boundary) vertices of $\mbox{Vertices}(\Graph_{S_0})$.

\begin{defi} \label{defIntBound}
\emph{There are triples of edges $e_i, e_j, e_k \in $
Edges($\Graph_{S_0}$) such that their union is a triangle
$T_{ijk}$. We define the following subset of the vertices of
$\Graph_{S_0}$, called the \emph{boundary vertices}.
$$
V_B = \{p \in \mbox{Vertices}(Graph_{S_0}) : p \text{ \mbox{is not
a vertex of two (or more) different triangles }} T_{ijk}\}.
$$
Note that $Y \subset V_B$. Also, denote
$$
V_B^c = \mbox{Vertices}(Graph_{S_0}) \setminus V_B.
$$
The subset $V_B^c$ is called the  \emph{interior vertices}.
}
\end{defi}

\begin{example}\emph{
The interior and boundary points, for the degeneration of the Hirzebruch
surface $F_{1,(2,2)}$:
\begin{center}
\epsfig{file=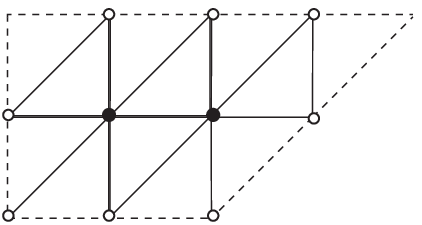}\\ \small{{Figure 1 :
 The white vertices are the boundary vertices $V_B$ and the black vertices are interior vertices $V_B^c$.}}
\end{center}
}
\end{example}

\begin{remark} \label{remRestrict}
\emph{We have two inequalities which relate the above constants.}

(1) \emph{ Assume that the degree of the ramification curve of $S$ is
$2\ell$ (which will be one of the conditions imposed on $S$. see condition \textbf{{(3)}} in Definition \ref{defSimDegSur}). we have that $2\ell \geq 2n-2$ (which follows from the
fact that $S$ is a ramified cover of $\pp^2$) or}
$$
n \leq \ell + 1.
$$

(2) \emph{Denote by $\overline{m}$ the number of vertices in
$\Graph_{S_0}$ (see Definition \ref{defGraphS0}), by
$\overline{\ell}$ the number of edges in $\Graph_{S_0}$ and
$\overline{n}$ the number of triangles in $\Graph_{S_0}$. Note that
$ n > \overline{n}$, $\overline{m} > m$ and $\overline{\ell} =
\ell$. By  the Euler characteristic for planar graphs we get
$\overline{m} - \overline{\ell} + \overline{n} = 1$ or}
$$
\overline{n} -1 < \ell  - m.
$$
\end{remark}

\subsection[Conditions on the planar degeneration]{Conditions on the planar degeneration} \label{subsecCondDeg}

In this subsection, we introduce the following conditions  that our projective surface $S$
has to satisfy.

\begin{defi} \label{defSimDegSur}
A surface $S = S_1$ is called \emph{simply--degeneratable surface} if it satisfies the following three conditions:
\end{defi}

\begin{condition*} \textbf{\emph{(1)}} 
\emph{$S$ admits a  planar degeneration, i.e., $\exists \rho : \tilde{S} \rightarrow \Delta$ s.t. $\rho^{-1}(1) = S_1 = S$ , $\rho^{-1}(0) = S_0$ and $S_0$ is a union of
planes.}
\end{condition*}

\begin{condition*} \textbf{\emph{(2)}}   
\emph{ The degeneration of $S$ to $S_0$ induces a degeneration of the
branch curve $B$ to $B_0$ that satisfies the following condition:
For a plane curve $C \subset \pp^2$, let $Sing(C)$ be the
singularities of $C$ w.r.t. a fixed generic projection $C \to \pp^1$.
Denote $Sing_0 \doteq Sing(B_0)$, $Sing_t =
Sing(B_t),\,t \neq 0$. For each $p \in Sing_0$ consider a small
enough neighborhood $U_p$ of $p$ as in Definition
\ref{defDeg}(iv). We require that the set $Sing(B) \setminus
\bigcup\limits_{p \in Sing_0} (U_p \cap Sing(B))$ contains only
simple branch points.}
\end{condition*}

\begin{condition*} \textbf{\emph{(3)}} 
\emph{The degeneration of the branch curves $B \to
B_0$ is two-to-one (see \cite{MoTe4},\cite{RobbT} for further details
on two-to-one degenerations of branch curves).}
\end{condition*}

\begin{remark}
\emph{We show that the three conditions above are independent.  As we are interested in planar degenerations, we look at the following examples when the degeneration already satisfies Condition \textbf{{(1)}}.\\
 (1) A degeneration of a smooth cubic surface in $\pp^3$ (whose branch curve is a sextic with six cusps) into a union of three planes, all of them intersecting in a line, is an example of a surface which does not satisfy Conditions
\textbf{{(2)}}, \textbf{{(3)}}.\\
(2) A degeneration of a union of three generic hyperplanes in $\pp^3$ (whose branch curve is a union of three lines, intersecting at three different points) into a union of three hyperplanes meeting at a single point, is an example of a degeneration that satisfies Condition \textbf{{(2)}} but not \textbf{{(3)}}.\\
(3) A degeneration of a cone over a smooth conic in $\pp^2$ into a union of two hyperplanes is an example of a degeneration that satisfies Condition \textbf{{(3)}} but not \textbf{{(2)}}.\\
(4) An example of planar degeneration that satisfies Conditions
\textbf{{(2)}}, \textbf{{(3)}} is a degeneration of a smooth quadric in $\pp^3$ into a union of two hyperplanes.}
\end{remark}

We  now define a fourth condition imposed on the degeneration: that every boundary vertex has at least
one interior vertex as a ``neighbor'' (see definition \ref{defIntBound}).

\begin{defi} \label{defEmbDegSur}
A surface $S$ is called \emph{embedded--degeneratable surface} if it is a simply--degeneratable surface and it  satisfies the following fourth condition:
\end{defi}

\begin{condition*} \textbf{\emph{(4)}} 
\emph{We require that for each boundary vertex $p \in V_B$ there exist an interior vertex $p^c \in V_B^c$ and an edge
$e_p \in $ Edges($\Graph_{S_0}$) s.t. $e_p$ connects $p$ and $p^c$.}
\end{condition*}

\begin{example}
\emph{The fourth condition is imposed in order to avoid degenerations as depicted in the following picture.
 Figure 2.[1] presents a degeneration with no interior points ($V_B^c = \emptyset$). Figure 2.[2] presents a degeneration with not enough neighboring interior vertices (though $V_B^c \neq \emptyset$).
 By definition, the dashed border lines
are  not a part of  Edges($\Graph_{S_0}$).
 \begin{center}
 \epsfig{file=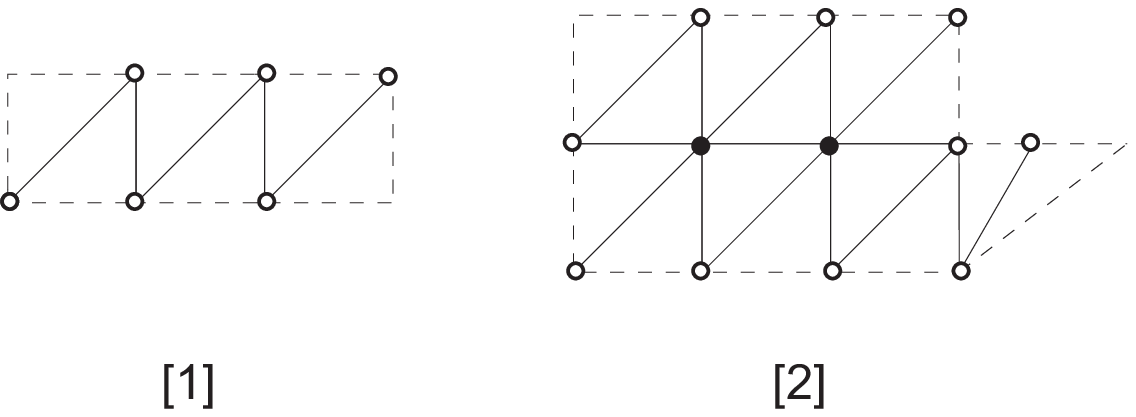}\\ \small{{Figure 2 : Degenerations which do no satisfy Condition \textbf{{(4)}}.\\ The white vertices are the boundary vertices $V_B$ and the black vertices are interior vertices $V_B^c$.}}
 \end{center}
 The following degeneration is a degeneration that satisfies all the four conditions.
 \begin{center}
 \epsfig{file=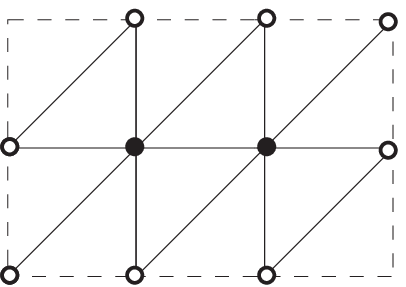}\\ \small{{Figure 3 : Allowable degeneration of $\pp^1 \times \pp^1$.\\
  The white vertices are the boundary vertices $V_B$ and the black vertices are interior vertices $V_B^c$.}}
 \end{center} }
\end{example}

\begin{remark}
\emph{Assume that there are interior vetrices in a given degeneration ($V_B^c \neq \emptyset$). Then in the case of a \emph{toric degeneration} any degeneration always satisfies Condition \textbf{{(4)}}. However, this is not known for general degenerations.}
\end{remark}

\subsection[Necessary condition on $\pi_1(\C^2 - B)$]{Necessary condition on $\pi_1(\C^2 - B)$} \label{subsec2MainThm}
 We present here the main result for this section -- under which conditions is $\pi_1(\C^2 - B)$  a quotient
 of  $\tB^{(2)}_n$. We begin with two examples:

\begin{example}\emph{
 For the degeneration of the Hirzebruch
surface $F_{1,(2,2)}$:
\begin{center}
\epsfig{file=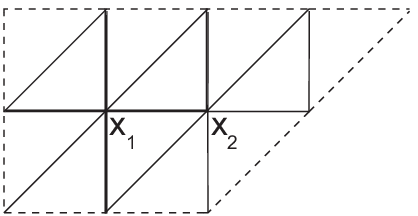}\\ \small{{Figure 4 : The degeneration
of  $F_{1,(2,2)}$. Note that the dashed border lines are
\\ not a part of the ramification curve}}
\end{center}
we have $Q = \{x_1,x_2\}$ (see \cite{Fr2} for the calculation of the local fundamental groups) and $m=2, \ell = 13, n= 12$, as depicted above. Note that in this case $\ell - m \leq n-1$.}\end{example}
\begin{example}\emph{For the degeneration of the surface $\CP^1 \times \mathbb{T}$
(where $\mathbb{T}$ is a torus), embedded with respect to the linear system
$(2,3)$
\begin{center}
\epsfig{file=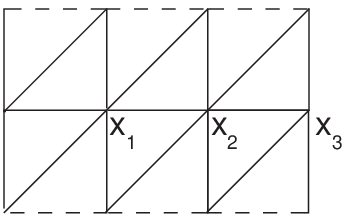}\\ \small{{Figure 5 : The degeneration
of  $(\CP^1 \times \mathbb{T})_{(2,3)}$. Note that the dashed
horizontal border lines are
\\ not a part of the ramification curve, and the vertical are. The vertical border lines are identified.}}
\end{center}
We have $m=3, \ell = 15, n= 12$ (as $Q = \{x_1,x_2,x_3\}$), and
the inequality $\ell - m \leq n-1$ is not satisfied. }
\end{example}

These observations lead us to state the following Theorem.
\begin{thm} \label{thmMainCondition}
Let $S$ be a smooth embedded--degeneratable projective surface. Let $B \subset \C^2$ its branch curve with respect to a generic projection, $B_0$ its degeneration. Denote $\ell = \frac{1}{2}\deg B, n =  \deg S, m = $ number of singular points $p$ of $B_0$ whose local fundamental group is a quotient of $\tB^{(2)}_{v(p)}$ (see Definitions \ref{defDeg}, \ref{defQ}).

If $\ell - m \leq n-1$ then there exist an epimorphism
$\tB^{(2)}_n \to G = \pi_1(\C^2 - B)$.
\end{thm}

The proof of this theorem will be given in Subsection \ref{subsecProof}.

\begin{example} \label{exampleKnownSur}
\emph{We give here a  list of known surfaces, satisfying Theorem \ref{thmMainCondition}:
$\pp^1 \times \pp^1$ embedded with respect to the linear system $al_1 + bl_2$, where $a,b>1$ (see \cite{Mo}),
 the Veronese surface $V_n,\,n \geq 3$ (see \cite{MoTe4}, \cite{MoTe5}), the Hirzebruch surfaces $F_1$ (embedded with respect to the linear system $aC + bE_0$, where $a,b>1,\, C, E_0$ generate the Picard group of $F_1$, see \cite{Fr2})
 and $F_2$ (embedded with respect to the linear system $2C + 2E_0$ (see \cite{AFT2}), and a few families of $K3$ surfaces (see \cite{FrTe}).}
\end{example}

Before proving the theorem, we want to review a few surfaces for
which the condition in Theorem \ref{thmMainCondition} does not
hold, presenting two conjectures.

For the first conjecture we need the following definition.

\begin{defi} \label{defArtinTilde}
\emph{(1) Given an Artin group $A$, generated by $\{x_i\}_{i=1}^r,\, r>2$, we define the following quotient:
$$
\widetilde{A} = A/\langle\!\langle
[x_2,x_3^{-1}x_1^{-1}x_2x_1x_3]\rangle\!\rangle.
$$
(2) Let $\deg$ be the following epimorphism: $\deg\,:\,A \to \Z$ s.t. $\deg(\prod x_i^{n_i}) = \sum n_i$.
Assume there exists an epimorphism from $A$ to the symmetric group $\sigma\,:\,A \to S_n$. In this case, define
$$\widetilde{A}^{(2)} \doteq \{(x,y) \in \widetilde{A} \times \widetilde{A}, \deg(x) =
\deg(y), \sigma(x) = \sigma(y)\}. $$}
\end{defi}

The  first  conjecture on the structure of $G = \pi_1(\C^2 -
B)$ is similar to Theorem~\ref{thmMainCondition}, when $Q \neq
\emptyset$ but does not contain enough points.

\begin{conjecture} \label{conjStructure}
\emph{ For a smooth embedded-degeneratable surface $S$ s.t. $|Q| = m
\geq 1$ and $\ell - m > n - 1 $  (i.e. does not satisfy the condition in Theorem
\ref{thmMainCondition})  one can associate a graph $T$ and an Artin
group $A(T)$ such that $G$ is a quotient of
$\widetilde{A(T)}^{(2)}$}
\end{conjecture}

The condition above means that $S$ has a planar degeneration with
2:1 degeneration of the branch curve, whose degeneration has
singular points in the set $Q$, but not enough. For example, See
\cite[Conjecture 3.7]{AFT} (on the embedding of $\CP^1 \times
\mathbb{T}$ with respect to the linear system $(m,n), m,n>1$) and
\cite{ATV} (for the degeneration of $\mathbb{T} \times
\mathbb{T}$).

We now review a few surfaces for which the set $Q$ is empty.

\begin{conjecture} \label{conjStructure2}
\emph{For a simply--degeneratable surface $S$ such that the set $Q
= \emptyset$ (i.e. the degeneration has only boundary points (see
Definition \ref{defIntBound})) and such that $G = \pi_1(\C^2 - B)$
has ``enough" commutation relations (see Remark \ref{remEn}), we
conjecture that one can associate a graph $T$ and an Artin group
$A(T)$ such that $G$ is a quotient of $A(T)$.}
\end{conjecture}

\begin{remark} \label{remEn}
\emph{Recall that $G$ has the natural monodromy epimorphism
$\varphi: G \to S_n$ ($n = \deg(S)$), defined by sending each
generator to a transposition, describing the sheets which are
exchanged. By ``enough" commutation relations we mean that for
$a,b \in G$ such that $\varphi(a), \varphi(b)$ are disjoint
transpositions, then $a,b$ commute.}
\end{remark}

\begin{example} \label{exampSurNotSat}
\emph{(1) The surface $\pp^1 \times \pp^1$ (embedded with respect
to the linear system $l_1 + bl_2, b\geq 1$ and denoted as $(\pp^1
\times \pp^1)_{(1,b)}$) and the Hirzebruch surface $F_1$ (embedded
with respect to the linear system $aC + E_0, a\geq 1$ and denoted
as $F_{1,(1,a)}$)  were investigated in \cite{AmS} and do not
satisfy condition \textbf{{(4)}} (see Definition \ref{defEmbDegSur}) and also the main condition in
Theorem \ref{thmMainCondition}. In both of these cases, however,
the fundamental group $\pi_1(\C^2 - B)$ is a quotient of the braid
group $B_n$, or equivalently a quotient of the Artin group $A(T)$,
where $T$ is depicted in the following figure.
\begin{center}
\epsfig{file=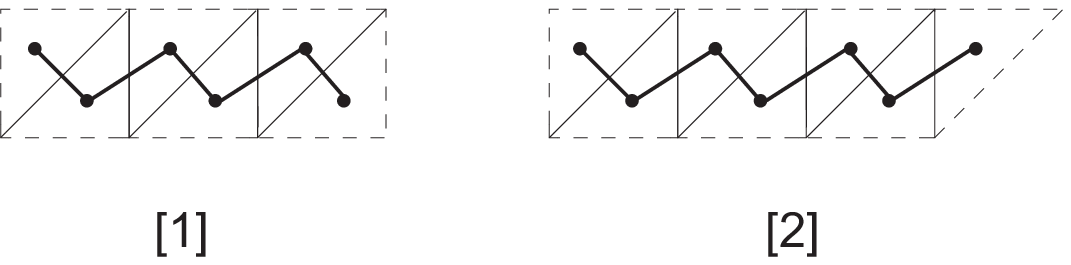}\\ \small{Figure 6 : The
degeneration of $(\pp^1 \times \pp^1)_{(1,3)}$ (figure [1]) and
$F_{1,(1,3)}$ (figure [2]) and their associated graphs $T$.}
\end{center}
(2) The Veronese surface $S = V_2 \subset \pp^5$ and its
associated fundamental group $\pi_1(\C^2 - B_S)$ were investigated
in \cite{MoTeGal}, \cite{Za}. Also in this example $V_2$ and its
degeneration do not satisfy the necessary conditions. Note that
this is an exceptional case to the previous example, as
$\pi_1(\C^2 - B_S)$ is not isomorphic to a quotient of $A(T)$,
where $T$ is depicted in the following figure.
\begin{center}
\epsfig{file=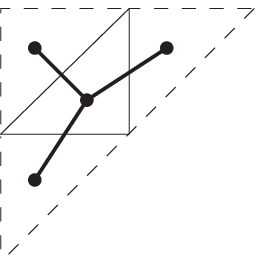}\\ \small{Figure 7 : The
degeneration of $V_2$
 and its associated graph $T$.}
\end{center}
This can be seen from \cite{Za}, as $\pi_1(\pp^2 - B_S)$ is
generated by four generators. The fact that $G = \pi_1(\C^2 -
B_S)$ does not have commutation relations is the reason we require
``enough" commutation relations in Conjecture \ref{conjStructure2}
(indeed, the condition  in Remark \ref{remEn} is not satisfied
w.r.t. the map $G \to S_4$).}

 \emph{Note that the Veronese surface $V_2$ is exceptional
also for other statements in classical algebraic geometry -- it
is, for example, the only counter example to the Chisini's
conjecture.}

\end{example}

\subsubsection[Virtual solvability of $G$]{Virtual solvability of
$G$}\label{subsecVirt} For surfaces whose planar degeneration
satisfy the condition introduced in Theorem
\ref{thmMainCondition}, the conjecture on the structure  (and the virtual solvability) of $G$
proposed in \cite{Te1} is correct. This is due to the fact that by
\cite[Remarks 3.7, 3.8]{denis}, if there is an epimorphism
$\tB_n^{(2)} \twoheadrightarrow G$, then $G$ is virtually
solvable. However, these conditions imply that the class of
embedded--degeneratable surfaces is rather small; for example, if
$\pi_1(S)$ contains a free group of rank 2, then $G$ is not
virtually solvable (see \cite{L}). These type of surfaces is the
main topic of Section 3.

 Also, by \cite[Corollary 4.9, Proposition 4.11]{L} one
can compute explicitly rank$(H_1(X_{Gal},\Z))$ (where $X_{Gal}$ is
the Galois cover of X. see subsection \ref{subSecGalCov}), and if
$X$ is simply connected, one can also find $\pi_1(X_{Gal})$.

\subsection[Proof of the main theorem]{Proof of the main theorem} \label{subsecProof}
 We first cite the Theorem we want to prove (Theorem \ref{thmMainCondition}):\\\\
\emph{
Let $S$ be a smooth embedded--degeneratable projective surface. Let $B \subset \C^2$ its branch curve with respect to a generic projection, $B_0$ its degeneration. Denote $\ell = \frac{1}{2}\deg B, n =  \deg S, m = $ number of singular points $p$ of $B_0$ whose local fundamental group is a quotient of $\tB^{(2)}_{v(p)}$ (see Definitions \ref{defDeg}, \ref{defQ}).}

\emph{If $\ell - m \leq n-1$ then there exist an epimorphism
$\tB^{(2)}_n \to G = \pi_1(\C^2 - B)$.\\\\}

\begin{proof}

We introduce the following
notation.

\begin{notation} \label{NotaDual}  \emph{Let $S_0 = \cup^n \Pi_i$ be the degeneration of $S$ as above, $R_0$ the degenerated ramification curve.
We build the graph $S^*_0 = (E,V)$ called the \textsl{dual graph
to }$S_0$ by the following procedure (see also \cite[pg. 532]{MoRoTe}).
each plane $\Pi_i$ corresponds to a vertex $v_i \in V,\, 1 \leq i
\leq n$, and each line $\Pi_k \cap \Pi_j = L_i \in R_0$ corresponds to
an edge $e_i \in E,\, 1 \leq i \leq \ell$, connecting the vertices
$v_k$ and $v_j$. For example
\begin{center}
\epsfig{file=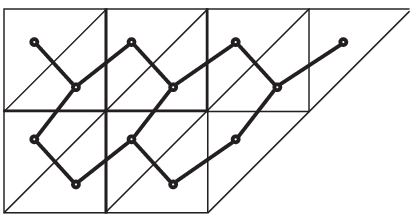}\\ \small{{Figure 8 : The dual graph
$S^*_0$ of the degeneration of   $F_{1,(2,2)}$.}}
\end{center}}
\end{notation}

We first prove the following lemma:
\begin{lemma} \label{lemSubTree}
There exists a spanning subtree of $S^*_0$ with $\ell - m$ edges if
and only if $\ell - m \leq n-1$.
\end{lemma}

\begin{proof}
If $n$ is the number of vertices in a connected graph, then if the
number of edges is greater than $n-1$, then there are cycles in
the graph. Therefore, if there is a spanning connected subtree of
$S^*_0$ with $\ell - m$ edges, then $\ell - m \leq n-1$.

For the other direction, assume first that $\ell - m = n-1$. For $ x \in \{p_i\}_{i=1}^{m'}$ we
denote by $L_x$ the set of lines
such that $x$ lies on them, and let $L^*_x$ be the set of edges in
$S^*_0$ corresponding to $L_x$. We create a new graph $T^*_0 =
(E_T, V_T)$ from $S^*_0$. The vertices of $T^*_0$ will be the same
as $S^*_0$, but for each $x \in Q$ we erase one edge $e_x$ from
$S^*_0$, such that $e_x \in L^*_x$. Since for each $x \in Q$,
$v(x)>3$, we demand that if there exist $x,y \in Q$ such that $x$
and $y$ are neighbors (i.e. there exist a line $L$ s.t. $x,y \in
L$), then $e_x \cap e_y = \varnothing$. We choose the $e_x$'s
satisfying the above requirements. Let us note that $m$ can be equal to $1$, so the choice of $y$ above is irrelavant.
 We now show that  the resulting graph $T^*_0$ is connected.

Note that if $x,y$ are neighbors, then locally the graphs $S_0$
and $S_0^*$ would look as in the following figure:

\begin{center}
\epsfig{file=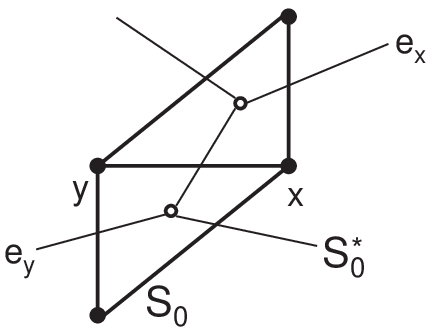}\\ \small{{Figure 9 : local
neighborhood of two vertices}}
\end{center}
since the degeneration is planar. Thus we can choose $e_x$ and
$e_y$ as depicted in Figure 7 and the resulting graph will be
connected. Now one can proceed by induction to prove
connectedness. Note that the number of edges in $T^*_0$ is $\ell - m$. Since $\ell - m = n-1,\, T^*_0$ is
a spanning subtree of $S^*_0$, by definition.

If $\ell - m < n-1$ there exist $k \in \mathbb{N}, k<m$ such that
$\ell - k = n-1$. We now choose $k$ points from $Q$, and proceed
as before to construct $T^*_0$.
\end{proof}

For example, the following figure presents a possible spanning
subtree $T^*_0$ for the degeneration of $F_{1,(2,2)}$:
\begin{center}
\epsfig{file=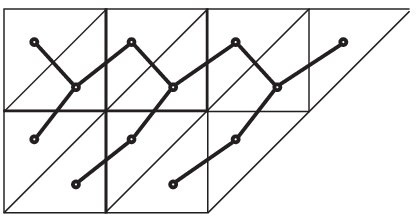}\\ \small{Figure 10}
\end{center}
By Lemma \ref{lemSubTree}, there exists a spanning subtree $T^*_0$.
We refine the construction  of $T^*_0$ in the following way. By
our assumptions, for each $x \in Q$, there exists an epimorphism
$\tB^{(2)}_{v(x)} \to G_x$, where $G_x$ is the local fundamental
group of $x$. As can be seen from Definition \ref{defBn},
$\tB^{(2)}_{v(x)}$ is generated by pairs $\{\G_i,
\G'_i\}_{i=1}^{v(x)-1}$, when the $\G_i$'s are the standard generators of $\tB_{v(x)}$. However, by the Van Kampen Theorem (see Theorem \ref{thmVK}), using the fact that the degeneration reduces the
degree of the branch curve by half (by Condition \textbf{{(3)}} on $S$. See Definition \ref{defSimDegSur}), we see
that $G_x$ is generated by pairs of (topological) generators $\{\g_i,
\g_{i'}\}_{i=1}^{v(x)}$. Thus, we can choose to express two generators
$\g, \g' \in \{\g_i,\g_{i'}\}_{i=1}^{v(x)}$ by the other generators s.t. the pair
$\g, \g'$ corresponds to a degenerated line $L \in
L_x$ and its corresponding edge $e_x \in L^*_x$ will be the edge
which we erase (possibly after renumeration of the generators of
$\tB^{(2)}_{v(x)}$ such that the erased edge will satisfy the
demands imposed on it as in Lemma \ref{lemSubTree}) in order to
get $T^*_0$.

\begin{remark} \label{remErase}
Note that  for all $x \in \{p_i\}_{i=1}^{m'}$ we erase at most
one edge from $L_x^*$.
\end{remark}
It is clear that for each $x \in Q$ there exists an embedding $\tB^{(2)}_{v(x)} \hookrightarrow G$.
Therefore $G_x \simeq \tB^{(2)}_{v(x)} / R_{v(x)} \hookrightarrow
G$ where $R_{v(x)} = \mbox{ker}(\tB^{(2)}_{v(x)} \to G_x)$.

\begin{remark} \label{remEmbPos}
\emph{The embedding $\tB^{(2)}_{v(x)} \hookrightarrow G$ might be possible only after a conjugation  of the
generators $\G_i, \G_{i'}$ by a certain power of  $\sigma_i$ (which is a generator of the braid group). See,
for example, \cite[Subsection 6.1.2]{denis}.}
\end{remark}

Let us now look at the points $x \in P \cup Y,\, x \not\in Q$:
these are the points whose corresponding local fundamental group
is not $\tB^{(2)}_{v(x)}$. We start, in the following subsection, with the most important case, and later we remark on two
more cases.

\subsubsection[Parasitic intersection points]{Parasitic intersection points} \label{subsecParasitic}
Each point $x \in B_0$ such that $v(x) = 2$ is an intersection of
two lines $l_i, l_j$. This kind of point, when $x \in P, x \not\in
\{p_i\}_{i=1}^{m'}$ is called a \emph{parasitic intersection
point}. These points are not a projection of singular points of
$R_0$, hence we get them as a result of the projection to $\CP^2$.
During the regeneration process (see Subsection
\ref{subsecRegRule}), each line is doubled, so eventually we get 4
nodes in $R$, and thus the local fundamental group is $\{\G_i,
\G_{i'}, \G_j, \G_{j'} :
[\G_{\underline{i}},(\G_{\underline{j}})_{\alpha}] = 1\}$, where
$\G_{\underline{i}} = \G_i$ or $\G_{i'}$ and $\alpha \in B_n$.
Examining these relations together, it can be seen easily that
$\alpha$ can be written as a product of generators which commute
with $\G_{\underline{i}}$ (see \cite[Thm IX.2.2]{MoTe1}, since
this arrangement of lines is a partial arrangement to what is
called in \cite[section IX, $\S 1$]{MoTe1} dual to generic).
Therefore, from the parasitic intersection points we induce the
commutator relations between different generators
$\G_{\underline{i}}, \G_{\underline{j}}$ such that the
corresponding lines $L_i, L_j$ do not have a vertex in common.

\begin{notation} \label{notParasitic} \emph{
Denote the set of all relations induced from the parasitic
intersection points as $R_{Par}$.}
\end{notation}

\begin{remark} \label{remTwoTypes}
\emph{Let us consider two more types of points which can appear during
the regeneration process:}
\begin{enumerate}
\item [{(I)}] \emph{First, recall that each $y \in Y$ is a 2-point: it is
on a line, which is the intersection of two planes. During the
regeneration process, this line is regenerated into a conic. If
$y$ is on the line $L_i$, whose corresponding edge in $T_0^*$ is
$e_i$, then we induce the relation $\G_i = \G_{i'}$ in $G$, where
$\{\G_i, \G_{i'}\}$ is the corresponding generators of $e_i$.
Explicitly, the local fundamental group is $\{\G_i, \G_{i'} : \G_i
= \G_{i'}\} \simeq \Z$. This is due to the fact that the line
$L_i$ is regenerated to a conic such that the branch point of the
conic (which corresponds to $y$) induces the relation $\G_i =
\G_{i'}$.}
\item [{(II)}] \emph{The second case is that $x \in \{p_i\}_{i=1}^{m'}, x
\not\in Q$ and thus $x$ is a projection of a singular point of
$R_0$ (if $x$ were not a projection of a singular point of $R_0$,
then the projection would not be a generic one). Note that
$v(x) > 2$. Let us assume that the local configuration of lines exiting from $x$ is as
in the following figure, when the lines are numerated by their order of appearance in the degeneration process:
\begin{center}
\epsfig{file=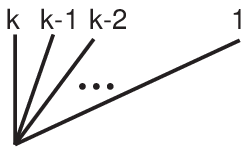} \\ \small{{Figure 11 : local
neighborhood of a $k$--point}}
\end{center}
In this case, the local braid monodromy factorization was
calculated in \cite{Fr1} and one can induce easily the local
fundamental group associated for this point(see e.g. \cite[Subsection 4.5]{MoRoTe}). Note that other numerations can appear also in
non-planar degenerations, such as in the degeneration of $\CP^1
\times C_g$ ($g\geq 1$. See Subsection \ref{secP1_C1} and Remark
\ref{remReg})}.
\end{enumerate}
\end{remark}

\begin{remark}
\emph{Recall that some of the singular points of a generic
projection $B \to \pp^1$ do not regenerate from $B_0$. By
Condition \textbf{{(2)}} on $S$ (see Definition \ref{defSimDegSur}), these singular points would be branch
points. These branch points only induce relations of the form
$\{\G_j = \G_{j'}\}$ when $\G_j, \G_{j'}$ correspond to the same
line $l_j$ in the degenerated branch curve $B_0$ (see \cite{RobbT}
for further explanations).}
\end{remark}

 We now examine what is the relation
between the local fundamental groups $G_x$ and the group $G$. It
is clear that for each $x$, $G_x \hookrightarrow G$, and in fact
$G \simeq (\underset{x \in P \cup Y}{\ast} G_x )/ \langle R_I
\rangle$ where $R_I$ is the identification of the same generators
in $G$ belonging to different $G_x$'s. Since we find a presentation of $G$ (and
resp. of the groups $G_x$) by means of the Van-Kampen theorem, we
can say that $G$ is generated by $2l$ (resp. $2v(x)$) generators.
However, by the definition of $Q$ and $T_0^*$ the number of generators for $G$
can be reduced to $2(l-m)$.

Let us examine two cases:
\begin{enumerate}
\item [{(i)}] Assume that $\ell-m = n-1$. By definition, for each $x
\in Q$, $G_x$ is isomorphic to a quotient of $\tB^{(2)}_{v(x)}$
(where this $G_x$ is generated by $2(v(x) - 1)$ generators
$\{\g_{x,i},\g_{x,i'}\}_{i=1}^{v(x)-1}$).
\begin{lemma} \label{lemGenG}
Let $\G \in G$ be a generator. So there exists $x \in Q$ s.t. $\G$
is a generator of $G_x$.
\end{lemma}
\begin{proof}
Assume that there is a generator $\G_0$ of $G$ such that it is not a
generator of $G_x$ for every $x \in Q$. This generator corresponds
to a line $l_0$ in $B_0$. By our construction, there are two
points $p_1,p_2$ on $l_0$ that belong to the set $P' \cup Y$, and by
assumption, both of them do not belong to $Q$ (recall that $P'$ is
the set of singular points of $B_0$ which are images of singular
points of $R_0$ and that $Y$ is the set of $2-$points).
We now look at two cases:\\
(I) One of the points belongs to $Y$.\\
Let $p_1 \in Y,\,p_2 \in P'$. The point $p_2$ is an ``inner" point (see Condition \textbf{{(4)}} in Definition \ref{defEmbDegSur}),
i.e., it does not lie on the border of the degenerated surface
$S_0$, as in this case $l_0$ would not induce a generator (recall that
we do not consider the border lines as part of $B_0$). Thus, the
local neighborhood of $p_1, p_2$ in $S_0^*$ looks as in the
following figure:
\begin{center}
\epsfig{file=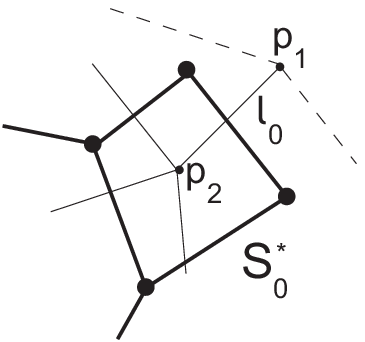}\\ \small{Figure 12 : Local
neighborhood of $p_1,p_2$.}
\end{center}
Since there is a spanning subtree $T_0^*$ (by Lemma \ref{lemSubTree}), one of the neighboring vertices to
$p_2$ has to be in $Q$, as otherwise, in the process of the construction of $T_0^*$,
we could not ``terminate" the circle $C$ whose center is the point $p_2$. Denote this vertex by $p_2^1$ and delete an edge from the circle $C$ (see the figure below).
\begin{center}
\epsfig{file=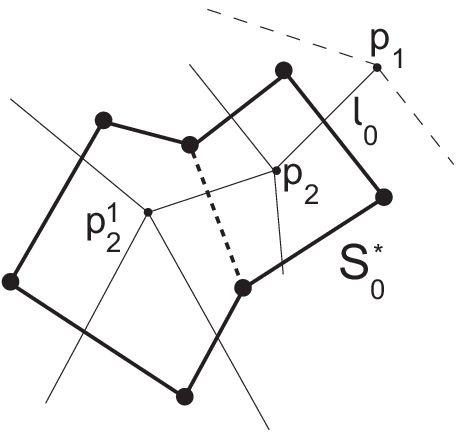}\\ \small{Figure 13 : Local
neighborhood of $p_1,p_2,p_2^1$.\\ The dashed edge is the erased
edge when trying to eliminate the circle containing $p_2$.}
\end{center}
However, now we have a new circle $C_1$ containing the points $p_2,p_2^1$.
Thus there is another point $p_2^2$ in $Q$, neighbor to $p_2$ or to $p_2^1$, as we have to
terminate the circle $C_1$, and we continue as above. But since this process is finite
(there are finite number of points in $Q$), eventually we couldn't erase one of the edges from the
circle $C_j$ (containing the points $p_2,p_2^1,...,p_2^j$). This is due to the fact that we would not find ``new" points in $Q$ s.t. one
of the corresponding edges to them can be erased. Thus we get a contradiction. \\
 (II) Two of the points belong to $P'$. We get a contradiction as in the first case, since now we have two circles
$C$ and $C'$, each around every point, which eventually could not
be resolved.
\end{proof}
 Thus the union of
all the generators of these $G_x$'s (s.t. we identify the same
generators in $G$) is the set of the $2(n-1)$ generators of $G$.
We know that
$$
G \simeq \bigg(\big(\underset{x \in Q}{\ast} G_x \big)/ \langle R_{I_Q} \rangle
\ast \big(\underset{x \not\in Q}{\ast} G_x \big)\bigg)/ \langle R_{rest}
\rangle,
$$
Where $R_{I_Q}$ ($R_{rest}$) is the set of relations
identifying identical generators in different local fundamental
groups for $x \in Q$ (resp. the set of the other relations, e.g., induced from identifying identical generators in different local fundamental
groups for $x \not\in Q$, or from the local fundamental groups of parasitic intersection points or from  extra branch points). But the generators of
$G$ are the generators of $\underset{x \in Q}{\ast} G_x$, and thus
$$
G \simeq \bigg(\big(\underset{x \in Q}{\ast} G_x \big)/ \langle R_{I_Q} \cup
R_{Par} \rangle \bigg) / \langle  R_{rest'} \rangle.
$$
Denoting $G_Q \doteq (\underset{x \in Q}{\ast} G_x )/ \langle
R_{I_Q} \cup R_{Par} \rangle$ it is enough to prove that there is
an epimorphism $\tB_n^{(2)} \twoheadrightarrow G_Q$.

Numerate the generators of $G_Q$ by $\{\G_i,\G_{i'}\}_{i=1}^{n-1}$ associated to the edges $E_T = \{t_i\}_{i=1}^{n-1}$ in the tree $T_0^*$, and let $\{x_i,x_{i'}\}_{i=1}^{n-1}$ be the generators of $\tB_n^{(2)}$. Define the epimorphic map
$$
\alpha : \tB_n^{(2)} \twoheadrightarrow G_Q,$$$$
x_i \mapsto \G_i, x_{i'} \mapsto \G_{i'}
$$
(possibly after conjugation. see Remark \ref{remEmbPos}).
We have to prove that the relations in $\tB_n^{(2)}$ hold in $G_Q$.
 Since $G_x \simeq \tB^{(2)}_{v(x)} / R_{v(x)}$ for each $x \in Q$ it is
 clear that the relations in $\tB_n^{(2)}$ of the form $aba = bab$ hold in $G_Q$.
 The commutator relations which are not induced from the commutator relations in $\tB^{(2)}_{v(x)},\,x \in Q$
 hold in $G_Q$ as the set of relations in $G_Q$ includes the set $R_{Par}$.

\item [{(ii)}]  Assume that $\ell-m < n-1$. Again, there exist $k \in
\mathbb{N},\, k < m$ such that $\ell-k = n-1$. Previously, in
Lemma \ref{lemSubTree}, we chose $k$ points from $Q$ to construct
$T^*_0$. Therefore we can continue as above. Note that by Remark
\ref{remErase}, even if the point $p_2$ (in Lemma \ref{lemGenG})
will have two neighboring vertices $\in Q$, we still could not
resolve the circle $C$.
\end{enumerate}

\end{proof}

\begin{remark} \label{remRestrict2}
\emph{Recall that for a degeneratable surface $S$ that satisfies all the conditions,
we denoted  $n =  \deg S,\,m = $ number of singular points $p$ of $B_0$ whose local fundamental group is a quotient of $\tB^{(2)}_{v(p)},$ and by $\overline{m}$ the number of vertices in
$Graph_{S_0}$ (see Definition \ref{defGraphS0}).
By the restrictions imposed by Remark \ref{remRestrict}
and Theorem \ref{thmMainCondition}, we can bound $\ell =
\frac{1}{2}$deg$B$. Explicitly, for $B$ to be a branch of
curve of degree $2\ell$ of a embedded--degeneratable surface s.t.
$G$ would be virtually solvable, the following inequalities should
be satisfied: }
\begin{equation} \label{eqnRest}
max(n,\overline{m}+n) < \ell + 1 \leq m+n.
\end{equation}
\end{remark}

\begin{remark}
\emph{As can be seen from  Subsection \ref{subsecVirt}, Example \ref{exampleKnownSur} and Remark \ref{remRestrict2}, the complete classification of smooth
surfaces whose planar degeneration
satisfy the condition introduced in Theorem
\ref{thmMainCondition} is not yet known, though some new restrictions are now clearer (e.g. inequality (\ref{eqnRest})). Moreover, \cite[Section 8]{CCFR2} has found some restrictions on surfaces admitting planar degeneration with some specific
conditions on the singularities of the degenerated surface. These conditions do shed some light on our class of  surfaces. For example, every singular point in the degenerated surface, denoted in \cite[Definition 3.5]{CCFR2} as  $E_m$-point ($m>3$), belongs to the set $Q$ (see Definition \ref{defQ}). Given a planar degeneration, \cite[Theorem 8.4]{CCFR2} imposes conditions on the square of the canonical class of the surface, when the degenerated surface has some specific singular points.  Certainly this theorem can be generalized to include more cases of singular points in the set $Q$ and
to the bigger classes of embedded--degeneratable surfaces. Moreover, \cite[Proposition 8.6]{CCFR2} states that for every surface there might be a birational model of it that is degeneratable into  a union of planes with mild singularities $p_i$ (s.t. the local fundamental group $G_{p_i}$ is known),
though it is not clear whether if this model is even simply--degeneratable (see Dentition \ref{defSimDegSur}).
 }

\emph{
 Note also that all the surfaces in Example \ref{exampleKnownSur} are simply connected, and this raises the conjecture that the desired class of surfaces is contained in the class of simply connected surfaces. Indeed, this is supported by that fact that if $S$ is a surface s.t.
$\pi_1(S)$ contains a free group of rank 2, then $S$ does not  satisfy the condition  in Theorem
\ref{thmMainCondition} (as $G$ is not virtually solvable). However, this is the subject of an ongoing research.
}
\end{remark}

\section[Non simply connected scrolls]{Non simply connected scrolls}
\label{sec3}

By \cite[Proposition 4.13]{L}, for a projective complex surface
$S$, if $\pi_1(S)$ is not virtually solvable, then $\pi_1(\CP^2 -
B)$ is not virtually solvable, where $B$ is the branch curve of
$S$ w.r.t. a generic projection. As  Liedtke \cite{L} points out,
for $S$ a ruled surface over a curve of genus $ > 1$ , $\pi_1(S)$
contains a free group of rank 2. Therefore, for such an $S$, there
does not exist a planar degeneration with enough ``good" singular
points (i.e. points in the set $Q$. See definition \ref{defQ}).
However, in the next section we examine what would be a possible
structure for $G = \pi_1(\C^2 - B)$ for such a surface.
Specifically, we consider the structure of this group when the set $Q$ is empty.

By Conjecture \ref{conjStructure}, the existence of
points in the set $Q$ would imply that $G$ would be a quotient of
$\widetilde{A}(T)^{(2)}$, where as in our case (see Thereom
\ref{thm1}), $G$ is a quotient of $A(T)$ (where $T$ is an
associated graph to the degeneration of $S$), as in Example \ref{exampSurNotSat}(1). This strengthens
Conjecture~\ref{conjStructure2}.

For the convenience of the reader, we begin with recalling the notions of the Braid Monodromy
Factorization (BMF) in subsection \ref{subsecBMF}. We then
investigate the surface $\pp^1 \times C_g$, where $C_g$ is a curve
of genus $g \geq 1$, and the corresponding fundamental group
$\pi_1(\C^2 - B_g)$, in subsections \ref{secP1_C1} and
\ref{secP1_Cg}. Using the results, we compute the fundamental
group of the Galois cover of these surfaces in subsection
\ref{subSecGalCov}.

\subsection[Background on Braid Monodromy Factorization]{Background on Braid Monodromy
Factorization} \label{subsecBMF}

Recall that computing the braid monodromy is the main tool to
compute fundamental groups of complements of curves. The reader
who is familiar with this subject can skip the following
definitions to Subsection \ref{secP1_C1}. We begin by defining the braid monodromy associated
to a curve.

Let $D$ be a closed disk in $ \mathbb{R}^2,$ \ $K\subset \Int(D),$
$K$ finite, $n= \#K$. Recall that the braid group $B_n(D,K)$ can
be defined as the group of all equivalent diffeomorphisms $\beta$
of $D$ such that $\beta(K) = K\,,\, \beta |_{\partial D} =
\text{Id}\left|_{\partial D}\right.$ (two
diffeomorphisms are equivalent
if they induce the same automorphism on $\pi_1(D - K,u)$). \\

\begin{defi} $H(\sigma)$ is a half-twist defined by
$\sigma$.\end{defi}

Let $a,b\in K,$ and let $\sigma$ be a smooth simple path in
$Int(D)$ connecting $a$ with $b$ \ s.t. $\sigma\cap K=\{a,b\}.$
Choose a small regular neighborhood $U$ of $\sigma$ contained in
$Int(D),$ s.t. $U\cap K=\{a,b\}$. Denote by $H(\sigma)$ the
diffeomorphism of $D$ which switches $a$ and $b$ by a
counterclockwise $180^\circ$ rotation and is the identity on
$D\setminus U$\,. Thus it defines an element of $B_n[D,K],$ called
{\it the half-twist defined by
$\sigma$ }.\\

Denote $[A,B] = ABA^{-1}B^{-1},\,\langle A,B\rangle =
ABAB^{-1}A^{-1}B^{-1}$. We recall  Artin's presentation of the
braid group:
\begin{thm} $B_n$ is generated by the half-twists $H_i$ of a sequence of paths
${\sigma_i}_{i=1}^{n-1}$ (such that $\sigma_i$ connected the $i^{th}$ and the $(i+1)^{th}$ points) and all the relations between $H_1,...,H_{n-1}$ follow
from:\begin{center} $[H_i,H_j] = 1\,\,$ if\,\,\,$
|i-j|>1$\\$\langle H_i,H_j\rangle = 1 \,\,if \,\,|i-j|=1$.
\end{center}
\end{thm}

Assume that all of the points of $K$ are on the $X$-axis (when
considering $D$ in $\mathbb{R}^2$). In this situation, if $a,b \in
K$, and $z_{a,b}$ is a path that connects them, then we denote it
by $Z_{a,b} = H(z_{a,b})$. If $z_{a,b}$ is a path that goes below
the $X$-axis, then we denote it by $\underline Z_{a,b}$, or just
$Z_{a,b}$. If $z_{a,b}$ is a path that goes above the $x$-axis,
then we denote it by $\overline Z_{a,b}$. We also denote by
$\overset{(c-d)}{\underline{Z}_{a,b}}$
($\underset{(c-d)}{\bar{Z}_{a,b}}$) the braid induced from a path
connecting the points $a$ and $b$ below (resp. above) the
$X$-axis, going above (resp. below) it from the point $c$ till
point $d$.\\

\begin{defi} {The braid monodromy
w.r.t. $C,\pi,u$} \emph{Let $C$ be a curve, $C\subseteq \C^2$ . Choose $O \in \C^2, O \not\in C$ such that the projection $f: \C^2 \to\C^1$  with center $O$ will be generic when restricting it to $C$. We denote $\pi = f|_C$ and
$\deg\pi = \deg\, C$ by $m.$ Let $N=\{x\in\C^1\bigm| \#\pi^{-1}(x)< m\}.$
Take $u\notin N,$
and let  $\C^1_u=f^{-1}(u).$  There is a  naturally defined homomorphism
$$\pi_1(\C^1-N,u)\xrightarrow{\vp} B_m[\C_u^1,\C_u^1\cap C]$$ which
is called {\it the braid monodromy w.r.t.} $C,\pi,u,$ where $B_m$
is the braid group. We sometimes denote $\vp$ by $\vp_u$}.\end{defi}

In fact,
denoting by $E$ a big disk in $\C^1$ s.t. $E \supset N$, we can
also take the path in $E\setminus N$ not to be a loop, but just a
non-self-intersecting path. This induces a diffeomorphism between
the models $(D,K)$ at the two ends of the considered path, where
$D$ is a big disk in $\C^1_u$, and $K = \C_u^1\cap C \subset
D$.

\begin{defi}  ${\psi_T \text{ the Lefschetz diffeomorphism induced by a path} \ T }$.
\emph{Let  $T$ be a path in $E\setminus N$ connecting $x_0$ with
$x_1$, $T:[0, 1]\ri E\setminus N$. There exists a continuous
family of diffeomorphisms $\psi_{(t)}: D\ri D,\ t\in[0,1],$ such
that $\psi_{(0)}=Id$, $\psi_{(t)}(K(x_0))=K(T(t)) $ for all
$t\in[0,1]$, and  $\psi_{(t)}(y)= y$ for all $y\in \partial D$.
For emphasis we write $\psi_{(t)}:(D,K(x_0))\ri(D,K(T(t))$. A
Lefschetz diffeomorphism induced by a path $T$ is the
diffeomorphism $$\psi_T= \psi_{(1)}: (D,K(x_0))\usr (D,K(x_1)).$$
Since $ \psi_{(t)} \left( K(x_{0})\right) = K(T(t))$ for all $t\in
[0,1]$, we have a family of canonical isomorphisms
$$\psi_{(t)}^{\nu}: B_p\left[ D, K(x_{0})\right] \usr B_p\left[
D, K({T(t)})\right], \ \quad \text{for all} \, \,
t\in[0,1].$$}\end{defi}

We recall Artin's theorem on the presentation of the Dehn twist of
the braid group as a product of braid monodromy elements of a
geometric-base (a base of $\p = \p(\C^1 - N, u)$ with certain
properties; see \cite{MoTe1} for definitions).\\
\begin{thm} Let $C$ be a curve transversal to the line in infinity, and
$\vp$ is a braid monodromy of $C , \vp:\p \rightarrow B_m$. Let
{$\delta_i$} be a geometric (free) base (called a g-base) of $\p,$
and $ \Delta^2$ is the generator of Center($B_m$). Then:
$$\Delta^2 = \prod\vp(\delta_i).$$ This product is also defined as
the \textsl{braid monodromy factorization} \emph{(BMF)} related to
a curve $C$.\end{thm}

Note that if $x_1,...,x_{n-1}$ are the generators of $B_n$, then
we know that $\Delta^2 = (x_1\cdot\ldots\cdot x_{n-1})^n$ and thus
deg($\Delta^2$) = $n(n-1)$.

So in order to find out what is the braid monodromy factorization
of $\Delta_p^2$, we have to find out what are $\vp
(\delta_i),\,\forall i$. We refer the reader to the definition of
a \textit{skeleton} (see \cite{MoTe2}) $\lambda_{x_j}, x_j \in N$,
which is a model of a set of paths connecting points in the fiber,
s.t. all those points coincide when approaching
$A_j=$($x_j,y_j$)$\in C$, when we approach this point from the
right. To describe this situation in greater detail, for $x_j \in
N$, let $x_j' = x_j + \alpha$. So the skeleton in $x_j$ is defined
as a system of paths connecting the points in $K(x_j') \cap
D(A_j,\varepsilon)$ when $0 < \alpha \ll \varepsilon \ll 1$,
$D(A_j,\varepsilon)$ is a disk centered
in $A_j$ with radius $\varepsilon$.\\

For a given skeleton, we denote by
$\Delta\langle\lambda_{x_j}\rangle$ the braid by rotates by 180
degrees counterclockwise a small neighborhood of the given
skeleton. Note that if $\lambda_{x_j}$ is a single path, then
$\Delta\langle\lambda_{x_j}\rangle = H(\lambda_{x_j})$.

We also refer the reader to the definition of $\delta_{x_0}$, for
$x_0 \in N$ (see \cite{MoTe2}), which describes the Lefschetz
diffeomorphism induced by a path going below $x_0$, for different
types of singular points (tangent, node, branch; for example, when
going below a node, a half-twist of the skeleton occurs and when
going below a tangent point, a full-twist occurs).

We define, for $x_0 \in N$, the following number:
$\varepsilon_{x_0} = 1,2,4$ when ($x_0, y_0$) is a branch / node /
tangent point (respectively). Explicitly, in local coordinates $(x,y)$ (where $(x_0, y_0)= (0,0)$), a branch is a singular point (w.r.t. the projection) with local equation $y^2=x$, a node -- $y^2=x^2$, and a tangent $y(y-x^2)=0$. So we have the following statement
(see \cite[Prop. 1.5]{MoTe2}):

Let $\gamma_j$ be a path below the real line from $x_j$ to $u$,
s.t. $\ell(\gamma_j)=\delta_j$. So
$$\vp_u(\delta_j) = \vp(\delta_j) =
\Delta \langle (\lambda_{x_j})\bigg(\prod\limits_{m=j-1}^{1}\delta_{x_m}\bigg)\rangle^{\varepsilon_{x_j}}.$$
When denoting $\xi_{x_j} =
(\lambda_{x_j})\bigg(\prod\limits_{m=j-1}^{1}\delta_{x_m}\bigg)$
we get
$$\vp(\delta_j)
= \Delta\langle(\xi_{x_j})\rangle^{\varepsilon_{x_j}}.$$ Note that
the last formula gives an algorithm to compute the needed
factorization. For a detailed explanation of the braid monodromy, see \cite{MoTe1}.\\

Assume that we have a curve $\bar{C}$ in $\CP^2$ and its BMF. Then
we can calculate the groups\\ $\pi_1(\CP^2 -\overline C)$ and
$\pi_1(\C^2 - C)$ (where  $C = \bar{C} \cap \C^2$). Recall that a
$g$-base is an ordered free base of $\p(D \backslash F,v)$, where
$D$ is a closed disc, $F$ is a finite set in Int($D$), $v \in
\partial D$ which satisfies several conditions; see \cite{MoTe1},
\cite{MoTe2} for the explicit definition.

Let $\{\G_i\}$ be a $g$-base of $G = \pi_1(\C_u-(\C_u \cap C),u),$ where $\C_u
= \C \times u$. We cite now the
Zariski-Van Kampen Theorem (for cuspidal curves) in order to
compute the relations between the generators in $G.$

\begin{thm} \label{thmVK} {\rm Zariski-Van Kampen (cuspidal curves version)} Let
$\overline C$ be a cuspidal curve in $\CP^2$. Let
$C=\C^2\cap\overline C.$ Let $\vp$ be a braid monodromy
factorization w.r.t. $C$ and $u.$ Let $\vp=\prod\limits_{j=1}^p
V_j^{\nu_j},$ where $V_j$ is a half-twist and $\nu_j=1,2,3.$

For every $j=1\dots p$, let $A_j,B_j\in\pi_1(\C_u-C,u)$ be such
that $A_j,B_j$ can be extended to a $g$-base of $\pi_1(\C_u-C,u)$
and $(A_j)V_j=B_j.$ Let $\{\G_i\}$ be a $g$-base of
$\pi_1(\C_u-C,u)$ corresponding to the $\{A_i, B_i \}$, where
$A_i, B_i$ are expressed in terms of $\G_i$. Then
$\pi_1(\C^2-C,u)$ is generated by the images of $\{\G_i\}$ in
$\pi_1(\C^2-C,u)$ and the only relations are those implied from
$\{V_j^{\nu_j}\},$ as follows:
$$\begin{cases} A_j\cdot B_j^{-1}&\quad\text{if}\quad \nu_j=1\\
[A_j,B_j]=1&\quad\text{if}\quad \nu_j=2\\
\langle A_j,B_j\rangle=1&\quad\text{if}\quad \nu_j=3.\end{cases}$$
$\pi_1(\CP^2-\overline C,*)$ is generated by $\{\G_i\}$ with the
above relations and one more relation $\prod\limits_i \G_i=1.$
\end{thm}

The following figure illustrates how to find $A_i, B_i$ from the
half-twist $V_i = H(\sigma)$:

\begin{center}
\epsfig{file=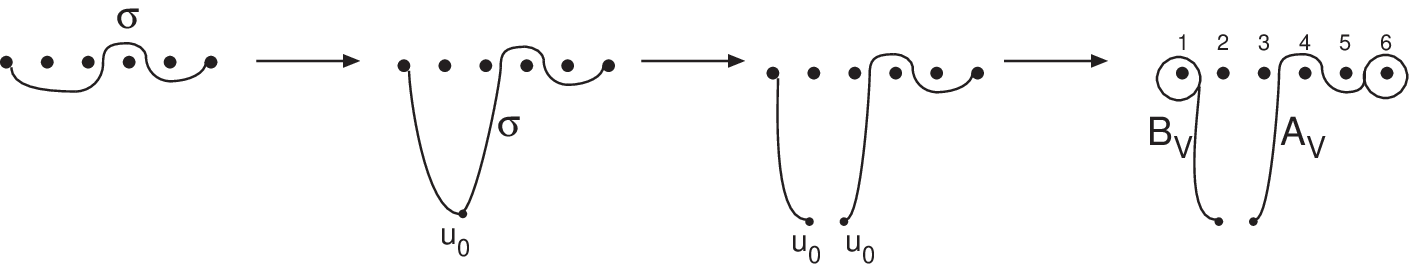}\\
Figure 14
\end{center}

So $$A_V = \G_4^{-1}\G_6\G_4,\,B_V = \G_1.$$

\subsubsection[Example of a BMF]{Example of a BMF}We give here an example of computing a simple Braid Monodromy
Factorization, for the following configuration:

\begin{center}
\epsfig{file=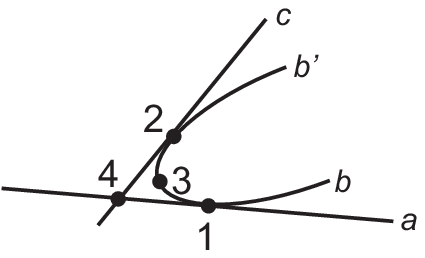}\\ Figure 15
\end{center}

We will need this factorization in Subsection \ref{secP1_C1}, where it will
be the factorization of the first regeneration a certain singular point.

\begin{prs} \label{prs2.1} The local braid monodromy factorization of the above
configuration is
$$\varphi = Z^4_{ab}Z^4_{b'c}\widetilde{Z}_{bb'}\widetilde{Z}^2_{ac}$$
where the braids $\widetilde{Z}_{bb'}, \widetilde{Z}_{ac}$
correspond to the following paths:

\begin{center}
\epsfig{file=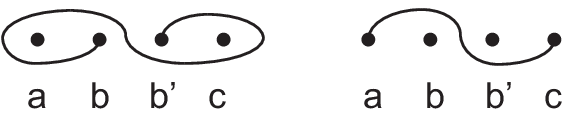}\\ \emph{Figure 16}
\end{center}

\end{prs}

\begin{proof} Let $\{p_j\}^4_{j=1}$ be the singular
points of the above configuration with
respect to $\pi_1$ (the projection to the $X$-axis) as follows:\\
$p_1, p_2$ - the tangent points of the parabola and the lines $L_a, L_c$ (denoted by $a$ and $c$ in Fig. 15).\\
$p_3$ - the branch point of the parabola.\\
$p_4$ - the intersection point of $L_a, L_c$.

Let $E$ (resp. $D$) be a closed disk on the $X$-axis (resp.
$Y$-axis).  Let $N = \{x(p_j) = x_j | 1 \leq j \leq 4\},$ s.t. $N
\subset E - \partial E$.  Let $M$ be a real point on the $x$-axis,
s.t. $x_j \ll M, \forall x_j \in N, 1 \leq j \leq 4$.  There is a
$g$-base $\ell(\g_j)^4_{j=1}$ of $\pi_1(E - N,u)$, s.t. each path
$\g_j$ is below the real line and the values of $\vp_M$ with
respect to this base and $E \times D$ are the ones given in the
proposition.  We look for $\vp_M(\ell(\g_j))$ for $j = 1, \cdots ,
4$.  Choose a $g$-base $\ell(\g_j)^4_{j=1}$ as above and put all
the data in the following table:

\begin{center}
\begin{tabular}{l|c|c|c} $j$ & $\lm_j$ & $\varepsilon_j$ & $\dl_{
j}$ \\ \hline
1 & $\langle a,b\rangle$ & 4 & $\Dl^2\langle a,b\rangle$\\
2 & $\langle b',c\rangle$ & 4 & $\Dl^2\langle b',c\rangle$\\
3 & $\langle b,b'\rangle$ & 1 & $\Dl^{1/2}_{IR}\langle b\rangle$\\
4 & $\langle a,c\rangle$ & 2 & $-$\\

\end{tabular}
\end{center}

So, we get the following:\\
$\xi_{x_1} = z_{a,b}\,,\,\varphi_M(\ell(\gamma_1)) = Z^4_{ab}$\\[1ex]
$\xi_{x_2} = z_{b',c}\,,\,\varphi_M(\ell(\gamma_2)) = Z^4_{b'c}$\\[1ex]
$\xi_{x_3} = \epsfig{file=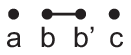}
\xrightarrow[\Delta^2<a,b>]{\Delta^2<b',c>}\,
\epsfig{file=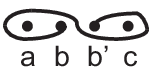},\,\varphi_M(\ell(\gamma_3)) =
\widetilde{Z}_{bb'}$\\[1ex]
$\xi_{x_4} = \epsfig{file=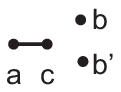}
\xrightarrow{\Dl^{1/2}_{IR}<b>} \epsfig{file=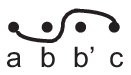}
\xrightarrow[\Delta^2<a,b>]{\Delta^2<b',c>}\,\epsfig{file=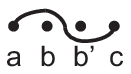}
,\,\varphi_M(\ell(\gamma_4)) = \widetilde{Z}_{ac} $\end{proof}

\subsubsection[Regeneration rules]{Regeneration rules}\label{subsecRegRule} We finish this subsection with the regeneration
rules. Given a degeneration $\rho : S \rightarrow \Delta$,  the
regeneration rules explain how the braid monodromy factorization
 of the branch curve of $S_0$ (under
generic projection) changes when passing to the braid monodromy
factorization  of the branch curve of $S_t$,\,$t \neq 0$. The
rules are (see \cite{MoTe4}, pp. 336-337):
\begin{enumerate}
\item \textbf{First regeneration rule}: The regeneration of a
branch point
of any conic:\\
A factor of the braid monodromy of the form $Z_{i,j}$ is replaced
in the regeneration by $Z_{i',j}\cdot
\overset{(j)}{\underline{Z}}_{i,j'}$
\item \textbf{Second regeneration rule}: The regeneration of a node:\\
A factor of the form $Z^2_{ij}$ is replaced by a factorized
expression $Z^2_{ii',j} := Z^2_{i'j}\cdot Z^2_{ij}$ ,\\
$Z^2_{i,jj'} := Z^2_{ij'}\cdot Z^2_{ij}$ or by $Z^2_{ii',jj'} :=
Z^2_{i'j'}\cdot Z^2_{ij'}Z^2_{i'j}\cdot Z^2_{ij}$. \item
\textbf{Third regeneration rule}: The regeneration of a tangent
point:\\
A factor of the form $Z^4_{ij}$ in the braid monodromy factorized
expression is replaced by\\ $Z^3_{i,jj'} :=
(Z^3_{ij})^{Z_{jj'}}\cdot (Z^3_{ij}) \cdot
(Z^3_{ij})^{Z^{-1}_{jj'}}$.
\end{enumerate}

\subsection[The fundamental group related to $\CP^1 \times C_1$]{The fundamental group related to $\CP^1 \times C_1$}
\label{secP1_C1}

We start by analyzing the degeneration of the surface $\CP^1
\times C_1$, where $C_1$ is a smooth curve of genus 1. Although
this surface was already investigated in \cite{AFT}, we present
here a different degeneration, which can be generalized to
surfaces of the form $\CP^1 \times C_g$ ($C_g$ is a smooth
genus--$g$ curve). This generalization will be discussed in the
next subsection but we give here a rough description of how this degeneration is done. See also Construction \ref{cons2}.

\begin{construction} \label{cons1} \emph{We review the degeneration described in \cite{CCFR}. Let $C$ be a smooth, rational normal curve of degree $n$ in $\pp^n$.
Since $C$ degenerates to a union of $n$ lines $l_i$ (s.t. $l_i \cap l_{i+1} = pt.$ for $1 \leq i  \leq n-1$,$l_i \cap l_j = \emptyset$ for $|i-j|>1$), the smooth rational
normal scroll $S = C \times \pp^1 \subset \pp^{2n+1}$ degenerates to surface $S` = \bigcup_{i=1}^n S_i$ such that each $S_i$ is a quadric (i.e. isomorphic to $\pp^1 \times \pp^1$). Each quadric $S_i$ meets $S \ S_i$ either along one or two lines of the same
ruling. Thus each quadric $S_i$ degenerates to the
union of two planes meeting along a line $l_i$, leaving the other line(s) fixed. Therefore, in $\pp^{2n+1}$,
the scroll $S$ degenerates  to a planar  surface $S''$ of degree $2n$. Assume $n>2$
Choose now two disjoint lines $\ell_1$, $\ell_4$ in the planes $S_1$ and $S_4$ such that $S_1 \cap S_2 \cap S_3 \not\in \ell_1$, $S_3 \cap S_4 \cap S_5 \not\in \ell_4$. As $\ell_1,\ell_4$ are skew, they span a $\pp^3$ which we denote as $\Pi$, such that $\Pi \cap S'' = \ell_1 \cup \ell_4$. Thus there exists a smooth quadric $Q$ in $\Pi$ such that
$\ell_1,\ell_4$ are lines of the same ruling on $Q$ and $Q \cap S'' = \ell_1 \cup \ell_4$. There, in $\Pi$, $Q$ degenerates to two planes $P_1,P_4$ s.t. $\ell_i \in P_i$. In \cite[Construction 4.2]{CCFR} one proves that  the planar surface $S'' \cap P_1 \cap P_4$ is indeed a degeneration $\CP^1
\times C_1$. See Figure 17 for the final degeneration
when $\pp^1 \times \pp^1$ is embedded w.r.t. the linear system $(1,3)$ (i.e. $n=3$ in the above notation).}
\end{construction}

%

\begin{center}
\epsfig{file=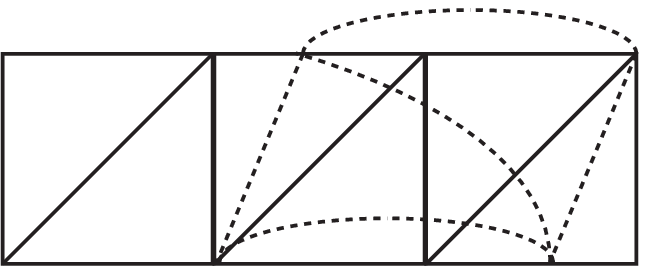} \\\small{Figure 17 : the degeneration
of $\CP^1 \times C_1$}
\end{center}
The dashed lines represent the attached degenerated quadric.
Some of the planes are intersecting other planes in lines. We
numerate (according to \cite{MoTe1}) the singular points of this arrangement of lines by $V_i, 1\leq
i\leq 8$ and
 the lines of intersection by $L_i, 1\leq i\leq 8 $, as follows:\\

\begin{center}
\epsfig{file=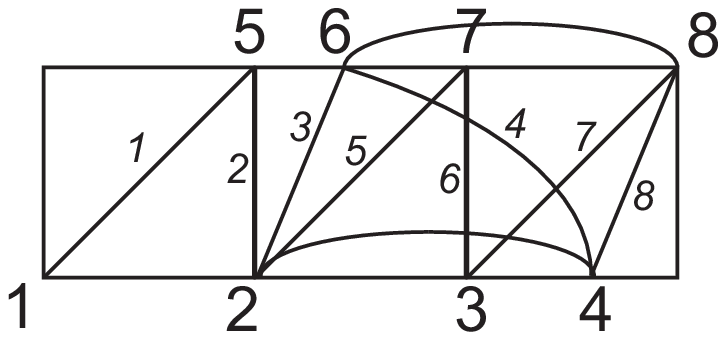}\\ \small{Figure 18 : numeration of
singular points and lines of the degenerated ramification curve of
$\CP^1 \times C_1$}
\end{center}

Note that $\cup L_i$ is the degenerated ramification curve $R_0$ with respect to a generic projection. Thus,
projecting the degenerated surface to $\CP^2$,
we denote by $B_0$ the (degenerated) branch curve and the images of $V_i$ by $v_i$.
We numerate the lines composing $B_0$ as before. Note that we have new
singular points, beside the points $v_i$, called parasitic intersection (see subsection \ref{subsecParasitic}). These points are
created from lines that did not intersect in $\CP^9$ but do
intersect in $\CP^2$. The braid monodormy factorization of the
degenerated branch curve is known to be (see \cite{MoTe1})
$\Pi_{i=8}^1 \widetilde{C_i}\Delta^2_i$, where $\widetilde{C_i}$
denotes the local braid monodromy factorization around
the parasitic intersection points and $\Delta^2_i$ the local
braid monodromy factorization around the point $v_i$. One can find the
$\widetilde{C_i}$'s according to \cite[Theorem IX]{MoTe1}.

\begin{remark} \label{remDeltaSq}
\emph{Since the regeneration of $\Delta^2_i$ for the different
points was already done, we give here references to the final
results.
 The points $v_i,\,\, i=3,..,8$ are 3--points.
(i.e., they are the images of the points $v_i$ which are locally the intersection of three planes. see e.g.,
\cite{MoTe4}) The factors that they contribute to the factorization (i.e. their local BMFs) are either
$Z^{(3)}_{a',bb'}\T \widetilde{Z}_{aa'}$ or $ Z^{(3)}_{aa',b}\T
\widetilde{Z}_{bb'}$ (where $v_i = L_a \cap L_b$).
 The point $v_1$ is a 2--point and contributes to the
factorization the factor $Z_{1,1'}$ (see Notation \ref{defM} and Remark \ref{remTwoTypes}(I)).
 For the point $v_2$, see a
more explicit explanation in the next remark.}
\end{remark}

\begin{remark} \label{remReg}
\emph{In a small neighborhood of $v_2$, the first line that
regenerates is $L_3$, which turns into a conic (see \cite{MoTe4}).
The braid monodromy factorization of this first regeneration is presented in
Proposition \ref{prs2.1}. In the following
regenerations we use the regeneration rules (see Subsection~\ref{subsecRegRule}): the tangent points (i.e., a braid of the form
$Z^4_{...}$) are regenerated into three cusps (three braids of the
form $Z^3_{...}$) and a node (a braid of the form $Z^2_{...}$)
into four nodes. Explicitly, the factorization
$Z^4_{23}Z^4_{3'5}\widetilde{Z}_{33'}\widetilde{Z}^2_{25}$ is
replaced by the factorization $Z^{(3)}_{22',3} Z^{(3)}_{3',55'}
\widetilde{Z}_{33'} \overset{(3)}{\lz^2}_{22',55'}$.}
\end{remark}

\begin{notation} \emph{Denote $\varphi(a,b,c) = Z^{(3)}_{aa',b} Z^{(3)}_{b',cc'}
\widetilde{Z}_{bb'} \overset{(b)}{\lz^2}_{aa',cc'}$ where
$\widetilde{Z}_{bb'}$ is as Figure 16, when the points $a$ and $c$
are doubled.}
\end{notation}

\begin{notation} \emph{$B_1 = $ the branch curve of $\pp^1 \times C_1$ embedded in $\pp^9$ w.r.t. a generic projection.}
\end{notation}

From Remarks \ref{remReg} and \ref{remDeltaSq}, we can induce the BMF of $B_1$:
\begin{thm} \label{thmBMF1}
The braid monodromy factorization of the branch curve $B_1$ of a generic
projection of $\CP^1 \times C_1$ embedded in $\CP^9$ is:
$$\Delta^2 = \prod_{i=8}^1 {C_i}\cdot H_i,$$ where\\

\begin{center}
$C_i = id,\, i=1,5,..,8, \quad C_2 = D_3\T D_5, \quad C_3 = D_6 \T
D_7, \quad C_4 = D_4\T D_8 $
\end{center}
 where
\begin{center}

$D_3 = \uz^2_{11',33'}, \quad D_4 = \lz^2_{11',44'}\T
Z^2_{22',44'}, \quad D_5 = \uz^2_{11',55'}\T Z^2_{44',55'} , \quad
 D_6 = \Pi_{i=1}^4 \underset{(5-5')}{\uz^2_{ii',66'}},$\\
 $D_7 = \Pi_{i=1}^5 \uz^2_{ii',77'}, \quad D_8 = \Pi_{\stackrel{i=1,2,}{\scriptscriptstyle{3,5,6}}}
 \uz^2_{ii',88'}\quad $
\end{center}
 and
\begin{center}
 $H_1 = Z_{1,1'},\quad H_i = Z^{(3)}_{a',bb'}\T \widetilde{Z}_{aa'}$
 for $i=4,5,7,8,\, H_i = Z^{(3)}_{aa',b}\T \widetilde{Z}_{bb'}$
 for $i=3,6$
\end{center} (when $v_i = L_a \cap L_b,\, a<b$), where
 $\widetilde{Z}_{\T,\,\T} $ is the braid induced from the following
 motion:\\\\
 $\widetilde{Z}_{aa'}:$
\begin{center}
\epsfig{file=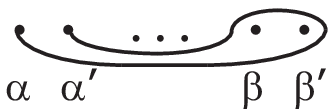}
\end{center}

\noindent $\widetilde{Z}_{bb'}:$
\begin{center}
\epsfig{file=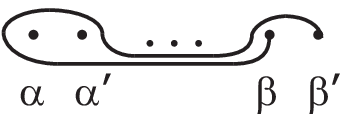}
\end{center}

$$Z^{(3)}_{a',bb'} = \prod_{q=-1,0,1}(Z^3_{a',b})_{Z^q_{b,b'}},\quad
 Z^{(3)}_{aa',b} = \prod_{q=-1,0,1}(Z^3_{a',b})_{Z^q_{a,a'}}$$
 and $H_2 = \varphi(2,3,5)$ where
 $ \widetilde{Z}_{33'}$ (a factor in the factorization $H_2$) is the braid induced from the following motion:

 \begin{center}
\epsfig{file=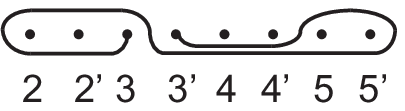}
\end{center}

\end{thm}

We recall the definition of an equivalence relation on the braid
monodromy factorization. Let $H$ be a group.
\begin{defi} [Hurwitz moves]
\emph{Let $\vec t= (t_1,\ldots ,t_m)\in H^m$\,. We say that $\vec
s =(s_1,\ldots ,s_m)\in H^m$ is obtained from $\vec t$ by the
Hurwitz move $R_k$ (or $\vec t$ is obtained from $\vec s$ by the
Hurwitz move $R^{-1}_k$) if
$$
s_i = t_i \quad\text{for}\  i\ne k\,,\, k+1\,,\\
s_k = t_kt_{k+1}t^{-1}_k\,,\\
s_{k+1} =t_k\,.
$$}
\end{defi}

\begin{defi}[Hurwitz move on a factorization] \emph{Let $H$ be a group $t\in H.$  Let
$t=t_1\cdot\ldots\cdot t_m= s_1\cdot\ldots\cdot s_m$ be two
factorized expressions of $t.$ We say that $s_1\cdot\ldots\cdot
s_m$ is obtained from $t_1\cdot\ldots\cdot t_m$ by a Hurwitz move
$R_k$ if $(s_1,\ldots ,s_m)$ is obtained from $(t_1,\ldots ,t_m)$
by a Hurwitz move $R_k$\,.}
\end{defi}

\begin{defi}
\emph{(1) Two factorizations are Hurwitz equivalent if they are
obtained
from each other by a finite sequence of Hurwitz moves.\\
(2) Let $g = g_1\T ... \T g_n$ be a factorized expression in a
group $H$ ($g_i \in H$), and denote by $()_h$ the conjugation by
$h \in H$. We say that $g$ is \emph{invariant} under $h$ if $g_h
\doteq (g_1)_h \T ... \T (g_n)_h$ is Hurwitz equivalent to $g$.}
\end{defi}

Let us examine the invariance relations on the braid monodromy
factorization from Theorem~\ref{thmBMF1}. From \cite{MoTe4} we
know that the expressions $C_i, 1\leq i \leq 8$ and $H_j, 1\leq j
\leq 8, j\neq 2$ are invariant under $Z^q_{kk'}, q \in \Z,
k=1,4,6,7,8$. Recall also that the expressions of the form
$Z^2_{ii',jj'}$ are invariant under $Z^p_{ii'}Z^q_{jj'}$ and
$Z^{(3)}_{i,jj'}$ is invariant under $Z^k_{jj'}$\,($k,p,q \in
\Z$). Note that if $\sigma \cap [j,j'] = \emptyset$ (where
$\sigma$ is a path in a disc containing the points $j,j'$ and
$[j,j']$ is a line connecting $j$ and $j'$) then $H(\sigma)$ is
invariant under $Z^k_{jj'}$\,($k \in \Z$).

\begin{remark}
\emph{Using these rules, we see that $H_2$ is invariant under
$Z^p_{22'}Z^q_{55'}$, and therefore the whole factorization is
invariant under $Z^{p_1}_{11'}Z^{p_2}_{22'}\Pi_{j=4}^8
Z^{p_j}_{jj'},\, p_j \in \Z$.}
\end{remark}

As was explained, during the regeneration process, every generator
$\G_j$ is doubled into two generators: $\G_j$ and $\G_{j'}$, so
$\p(\C^2-B_1)$ is generated by $\{\G_j,\G_{j'}\}_{j=1}^8$. From now
on, we denote the generator $\G_j$ by $j$ and the generator
$\G_{j'}$ by $j'$. Let $\underline{j}$ denote $j$ or $j'$, and $e$
the unit element in $\p(\C^2-B_1)$.

\begin{notation}
$[a,b] = aba^{-1}b^{-1},\, <\!a,b\!> = abab^{-1}a^{-1}b^{-1},\,a_b
= b^{-1}ab$.
\end{notation}

\begin{prs} \label{prs1}
$G_1 \doteq \p(\C^2-B_1)$ is generated by $\{j,j'\}_{j=1}^8$ and has the
following relations:
\begin{enumerate}
\item [\emph{(1)}] $1=1'$
\item [\emph{(2)}] $\langle  \un 6,\un 7 \rangle = \langle  \un 4,\un 8 \rangle =\langle  \un 1,\un 2 \rangle =\langle  \un 3,\un 4 \rangle =
\langle  \un 5,\un 6 \rangle =\langle  \un 7,\un 8 \rangle = e$
\item [\emph{(3)}] $7' = 6^{-1}6'^{-1}76'6,\quad 4=8'84'8^{-1}8'^{-1},\quad 1=2'21'2^{-1}2'^{-1}, \quad \\
   4'=3^{-1}3'^{-1}43'3, \quad 5=6'65'6^{-1}6'^{-1},\quad 7=8'87'8^{-1}8'^{-1}$
\item [\emph{(4)}] $[\un 1, \un 3] = [\un 2, \un 4] = [\un 1, \un 4] = [\un 1, \un 5] = [\un 4, \un 5] = e\\$
$[\un i, \un 6] = e , 1 \leq i \leq 4, \quad [\un i, \un 7] = e , 1 \leq i \leq 5, \quad [\un i, \un 8] = e ,
 1 \leq i \leq 6, i \neq 4 $

\item [\emph{(5)}] $\langle  \un 2,3 \rangle = \langle  \un 5,3'  \rangle = e,\\$
$5'53'5^{-1}5'^{-1} = 32'232^{-1}2'^{-1}3^{-1}$,\\ $[3\un 2
3^{-1},\un 5] = e.$
\end{enumerate}

\end{prs}
\begin{proof} In the proof, we use the Van-Kampen theorem (Theorem \ref{thmVK}), the complex conjugation
method and the invariance relations. Relation (1) is induced from
the braid $Z_{11'}$. Relations (2) and (3) are induced, using
Van-Kampen and invariance, from the factors $H_i, 3 \leq i \leq
8$. Relations (4) are induced from the parasitic intersection
points -- the factors $C_i$. Relations (5) are induced from the
factors in $H_2$. \end{proof}

\begin{prs} \label{prs2} The following relations hold in $G_1$:
\begin{enumerate}
\item [\emph{(6)}] $\langle  \un 2,\un 3 \rangle = \langle  \un 3,\un 5 \rangle = \langle  \un
2,\un 5 \rangle = e$ \item [\emph{(7)}] $[\un 2^{-1}\un 3 \un 2, \un
5]=e$
\end{enumerate}
\end{prs}

\begin{proof} By Proposition \ref{prs1} ((5) and (3)), it is known that
$$e = \langle  3',\un 5  \rangle = \langle   4^{-1}34'3^{-1}4, \un 5  \rangle \underset{[4,\un 5]=e}{=} \langle   34'3^{-1}, \un 5  \rangle \underset{\langle  3, 4' \rangle = e}{=} \langle   4'^{-1}34',\un 5  \rangle \underset{[4',\un 5]=e}{=} \langle   3,\un 5  \rangle.$$ Thus $\langle  \un 3,\un 5  \rangle=e.$ Also, we have: $$e = \langle   \un 2, 3  \rangle = \langle   \un 2, 4'43'4^{-1}4'^{-1} \rangle \underset{[\un 4,\un 2]=e}{=} \langle   \un 2, 3' \rangle\quad  \Rightarrow \quad \langle  \un 2,\un 3  \rangle=e.$$
From relation (5) we get $3' =
5^{-1}5'^{-1}32'232^{-1}2'^{-1}3^{-1}5'5$ and also
$$e = \langle   3', 5  \rangle = \langle   5^{-1}5'^{-1}32'232^{-1}2'^{-1}3^{-1}5'5,5  \rangle = \langle  32'232^{-1}2'^{-1}3^{-1} , 5'55'^{-1} \rangle \underset{\mbox{\small{Invariance}}\, Z_{55'}}{=}$$$$ \langle   32'232^{-1}2'^{-1}3^{-1} , 5' \rangle \underset{\langle  2, 3  \rangle = e}{=} \langle   32'3^{-1}232'^{-1}3^{-1}, 5' \rangle \underset{[3\un 23^{-1}, 5']=e}{=} \langle   2 , 5' \rangle$$
and by invariance relations we get $\langle  \un 2,\un 5 \rangle = e.$ This completes the proof of (6).\\
From (5) we have $$e = [3\un 2 3 ^{-1}, \un 5 ] = [\un 2^{-1} 3
\un 2 ,\un 5 ]= [\un 2^{-1}4'43'4^{-1}4'^{-1} \un 2 ,\un 5 ]
\underset{[\un 4,\un 2]= [\un 4,\un 5]=e}{=} [\un 2^{-1} 3' \un 2
,\un 5 ].$$ Thus $[\un 2^{-1} \un 3 \un 2 ,\un 5 ] =
e.$\end{proof}

Our next task is to express the generators $j'$ ($j=1,2,3,5,..,8$) by the generators $1 \leq j \leq 8$ and $4'$.
This is easy: using (3), we get
\begin{enumerate}
\item [{(8)}] $1'=1,\quad 2' = 1^{-1}212^{-1}1,\quad 3' = 4^{-1}34'3^{-1}4, \quad 8' = 4^{-1}84'8^{-1}4 $\\
$7' = 8^{-1}8'^{-1}78'8, \quad 6' = 7^{-1}67'6^{-1}7, \quad 5' = 6^{-1}6'^{-1}56'6. $
\end{enumerate}

Therefore, the group $G$ is generated by the generators
$\{j\}_{j=1}^8 \cup \{4'\}$. We note that all the commutator and
triple relations (i.e., (2), (4), (6), (7)) that involve the
generators $j'$ where $j=1,2,3,5,..,8$ can be reduced, since these
$j'$`s are expressed in terms of the other generators. Our
task now is to reduce most of the relations coming from the branch
points, i.e. (3) and the second relation at (5). Notice that all
of the relations in (3) are already reduced, as we have used them
to define the generators $j'$ (by (8)). However, one can see that,
for example, in the second relation in (5) we can substitute the
generators $j'$ using (8), till we get an expression containing
only the generators $\{j\}_{j=1}^8 \cup \{4'\}$. Therefore, we get
the following relation:

\begin{enumerate}
\item [{(9)}]
$(4')_{3^{-1}45^{-1}6^{-1}7^{-1}8^{-1}4^{-1}84'^{-1}8^{-1}47^{-1}6^{-1}5^{-1}}
= (3)_{2^{-1}1^{-1}21^{-1}2^{-1}13^{-1}}.$
\end{enumerate}

\begin{notation} \label{NotPho}
\emph{Denote relation (9) by $\rho_1$.}
\end{notation}

 Note  that (9) can be described as a ``global" relation, involving
almost all the generators of the group. We need only to find out
what are the ``local" relations, involving only the
generators $4,4',3$ and $8$.
\begin{prs} \label{prs3} The following relations hold in $G_1$:
\begin{enumerate}
\item [\emph{(10)}] $\langle   8 4'  8^{-1}, 4 \rangle = \langle   3 4' 3^{-1},
4 \rangle = e.$
\item [\emph{(11)}] $[3^{-1}43,84'8^{-1}]=e.$
\end{enumerate}
\end{prs}

\begin{proof} Knowing that $4' = 3^{-1}3'^{-1}43'3 $
we see that:
$$\langle   3 4' 3^{-1}, 4 \rangle = \langle   33^{-1}3'^{-1}43'33^{-1}, 4  \rangle = \langle   3', 4 \rangle \underset{\mbox{rel.} (2)}{=} e. $$
The same is dome for the second relation, using $4' =
8^{-1}8'^{-1}48'8$. This proves relation set (10).

For the relation (11), we use the relation $4 = 8'84'8^{-1}8'^{-1}$.
$$
[3^{-1}43,84'8^{-1}] = [3^{-1}8'84'8^{-1}8'^{-1}3,84'8^{-1}] \underset{[3,\underline{8}]=e}{=}
[8'83^{-1}4'38^{-1}8'^{-1},84'8^{-1}] = $$
$$[8^{-1}8'83^{-1}4'38^{-1}8'^{-1}8,4']
\underset{\mbox{Inv. } Z_{8,8'}}{=} [8'3^{-1}4'38'^{-1},4'] = [3^{-1}4'3,8'^{-1}4'8']
\underset{\langle 3,\underline{4}\rangle=\langle 8',\underline{4}\rangle=e}{=} $$$$ [4'34'^{-1},4'8'4'^{-1}] = [3,8'] = e.
$$

\end{proof}

The last relation we want to induce concerns the fact that once
the we have two ``circles" in the graph associated to the
generators (see Figure 19 in Proposition \ref{prs4}), we ought to
find a triple relation relating each two edges that intersect in
one vertex.
\begin{prs} \label{prsCycleRel} The following relation holds in $G_1$:
\begin{enumerate}
\item [\emph{(12)}] $\langle   3^{-1}43 , 56784'8^{-1}7^{-1}6^{-1}5^{-1}  \rangle = e.$
\end{enumerate}
\end{prs}

\begin{proof}
First, we prove that $\langle   3^{-1}43 , 5  \rangle = e $.
$$
\langle   3^{-1}43 , 5  \rangle \underset{\langle 3,4\rangle=e}{=} \langle   434^{-1} , 5  \rangle \underset{[5,4]=e}{=} \langle   3 , 5  \rangle = e.
$$
Thus
$$
\langle   3^{-1}43 , 56784'8^{-1}7^{-1}6^{-1}5^{-1}  \rangle = \langle   5^{-1}\cdot (3^{-1}43)\cdot 5 , 6784'8^{-1}7^{-1}6^{-1}  \rangle =
\langle   3^{-1}43\cdot 5\cdot (3^{-1}43)^{-1} , 6784'8^{-1}7^{-1}6^{-1}  \rangle $$$$
\underset{[6,3]=[7,3]=[6,4]=[7,4]=e}{=} \langle   3^{-1}43\cdot (7^{-1}6^{-1}567)\cdot (3^{-1}43)^{-1} , 84'8^{-1}  \rangle
\underset{\mbox{rel. } (11)}{=} \langle    (7^{-1}6^{-1}567)\cdot , 84'8^{-1}  \rangle \underset{\langle 5,6\rangle=e}{=}  $$$$
\langle    7^{-1}565^{-1}7, 84'8^{-1}  \rangle \underset{[5,7]=[5,8]=[5,4']=e}{=}
 \langle    7^{-1}67 , 84'8^{-1}  \rangle \underset{\langle 6,7\rangle=[6,8]=[6,4']=e}{=} \langle    7 , 84'8^{-1}  \rangle =  \langle    7 , 8 \rangle = e.
$$

\end{proof}

\begin{defi} \label{def1}
\emph{Let $T$ be a graph with $n$ vertices. In the spirit of \cite{RTU} and \cite{ASM}, denote
by $\hat A(T)$ the following \textsl{generalized Artin group}. This is the group generated by the edges $u \in T$
subject to the following relations:}
\begin{enumerate}
\item [\emph{(i)}] $uv = vu$ \emph{if} $u, v$ \emph{are disjoint}.
\item [\emph{(ii)}] $uvu = vuv$ \emph{if} $u, v$ \emph{intersect
in one vertex}.
 \item [\emph{(iii)}] $[u, vwv^{-1}] = e$ \emph{for} $u, v, w
\in T$ \emph{which meet in only one vertex}.

\item [\emph{(iv)}]\emph{for}
$u, v, v', w \in T$ \emph{which intersect in the following way:}

\begin{center}
\epsfig{file=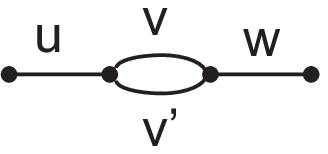}
\end{center}
the edges satisfy the relations:
\begin{enumerate}
\item [\emph{(1)}]
$\langle   w v'  w^{-1}, v \rangle = \langle   u v' u^{-1}, v \rangle = e$
\item [\emph{(2)}]
$ [u^{-1}vu,wv'w^{-1}] = e. $
\end{enumerate}
\item [\emph{(v)}] For two circles in the graph $T$, embedded in each other in the following way
\begin{center}
\epsfig{file=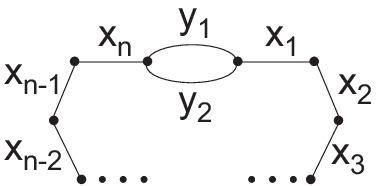}
\end{center}
The edges satisfy the relation:  $ \langle   x_n^{-1}y_1x_n ,x_{n-1}\cdot ...\cdot x_2x_1y_2x_1^{-1}x_2^{-1} \cdot ...\cdot x_{n-1}^{-1}  \rangle = e$.
\end{enumerate}

\end{defi}
 Summarizing propositions
\ref{prs1}, \ref{prs2}, relation (9), \ref{prs3} and \ref{prsCycleRel} we get the following
\begin{prs} \label{prs4}$G_1 \simeq  \hat A(T_1)/\rho_1$, where $T_1$
is the following graph:
\begin{center}
\epsfig{file=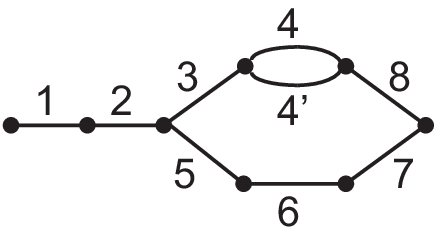}\\ \small{\emph{Figure 19}}
\end{center}

\end{prs}

\begin{remark} \label{rem1}
\emph{Let $T_1, T_2$ be connected disjoint graphs. Then $\hat
A(T_1 \cup T_2) = \hat A(T_1) \times \hat A(T_2)$.}
\end{remark}

\subsection[The fundamental group related to $\CP^1 \times C_g,\, g>1$]{The fundamental group related to $\CP^1 \times C_g,\, g>1$} \label{secP1_Cg}

In this subsection, we compute the BMF of the branch curve $B_g$
of $\CP^1 \times C_g,\, g>1$ and the corresponding fundamental
group. We show the connections between these groups and the
twisted Artin group defined earlier (see Definition \ref{def1}).
We begin with the surface $\CP^1 \times C_2$.

\begin{construction} \label{cons2}\emph{
As in Construction \ref{cons1}, we can build a degeneration of
$\CP^1 \times C_2$. Embedding the rational scroll
$\CP^1 \times \CP^1$ with respect to the linear system $(1,6)$, we degenerate it into $S''$ a union  of $12$ planes $S_i$.
Choosing two pairs of lines $\ell_1,\ell_4$ in $S_1,S_4$ and $\ell_7,\ell_{10}$ in $S_7,S_{10}$, we can attach to each pair a quadric $Q_j$, $j=1,7$ such that $Q_j \cap S'' = \ell_{j} \cup \ell_{j+3}$. Degenerating each of the two quadrics into two planes, the union of the $16$ planes is a degenerated planar surface which is the degeneration of
$\CP^1 \times C_2$, as is proved in \cite[Theorem 4.6]{CCFR}. See Figure 20 for the degeneration.}
\end{construction}

\begin{center}
\epsfig{file=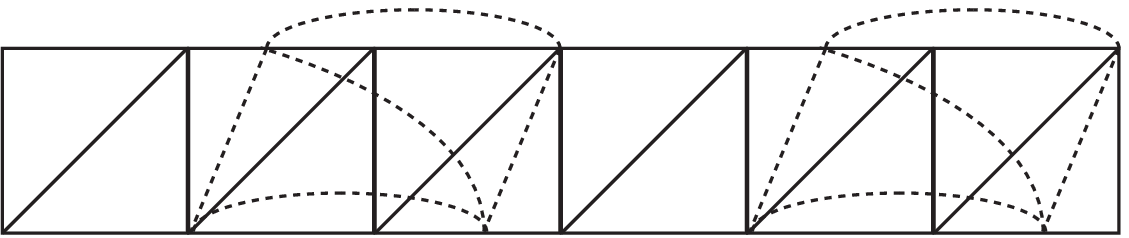}\\ \small{Figure 20 : Degeneration of
$\CP^1 \times C_2$}
\end{center}

Repeating the process described in the previous subsection, we
numerate the singularities $v_i,\, 1\leq i \leq 16$ of the
degenerated surface $\CP^1 \times C_2$ and the lines of
intersection $L_i,\, 1\leq i \leq 17$ as follows:

\begin{center}
\epsfig{file=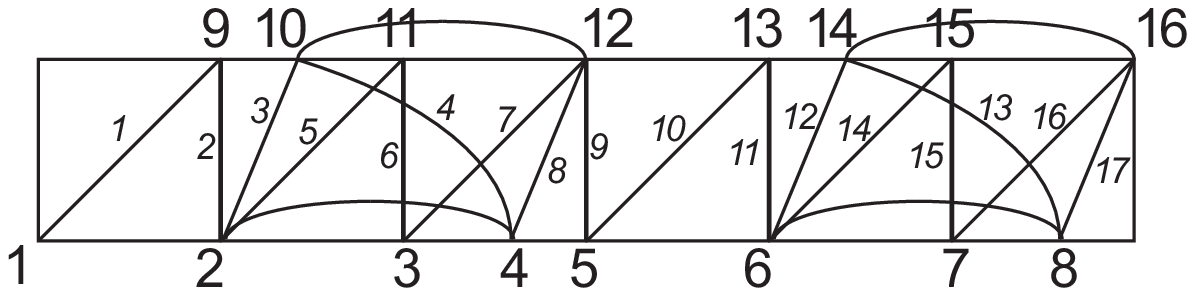}\\ \small{Figure 21}
\end{center}

Once again, we project the degenerated surface to $\CP^2$, compute
the BMF of the degenerated branch curve ($= \Pi_{i=16}^1
\widetilde{C_i}\Delta^2_i$) and regenerate it.
\begin{remark}
\emph{The points $v_i,\, 3 \leq i \leq 16,\, i\neq 6,12$ are
3-points, and their local regenerated BMFs are either
$Z^{(3)}_{a',bb'}\T \widetilde{Z}_{aa'}$ or $ Z^{(3)}_{aa',b}\T
\widetilde{Z}_{bb'}$ (where $v_i = L_a \cap L_b$). The point $v_1$ contributes the factor $Z_{1,1'}$ to the global
BMF. The local regenerated BMF of a
neighborhood of the points $v_2,v_6,v_{12}$ is computed as in
Remark \ref{remReg}.}
\end{remark}

\begin{thm}
The braid monodromy factorization of the branch curve $B_2$ of a
generic projection of $\CP^1 \times C_2$ embedded in $\CP^{17}$
is:
$$\Delta^2 = \prod_{i=16}^1 {C_i}\cdot H_i,$$ where\\
\begin{center}
$C_i = id,\, i = 1,9,..,16,\quad C_2 = D_3 \cdot D_5,\quad C_3 =
D_6 \cdot D_7,\quad C_4 = D_4 \cdot D_8,\quad C_5 = D_9 \cdot
D_{10}$
\end{center}
\begin{center}
 $C_6 = D_{11} \cdot D_{12} \cdot D_{14},\quad C_7 = D_{15}
\cdot D_{16},\quad C_8 = D_{13} \cdot D_{17}$
\end{center}
where
\begin{center}
$D_3 = \uzs_{11',33'},\quad D_4 =
\underset{(3-3')}{\uzs_{11',44'}}\underset{(3-3')}{\uzs_{22',44'}},\quad
D_5 = \Pi_{i=1}^2 \uzs_{ii',55'}\cdot Z^2_{44',55'},\quad D_6 =
\Pi_{i=1}^4 \underset{(5-5')}{\uzs_{ii',66'}}$\end{center}
\begin{center}$D_7 =
\Pi_{i=1}^5 \uzs_{ii',77'} ,\quad D_8 = \Pi_{i=1}^{3}
\underset{(7-7')}{\uzs_{ii',88'}}\cdot
\underset{(7-7')}{\uzs_{55',88'}}\uzs_{66',88'},\quad D_9 =
\Pi_{i=1}^6 \underset{(7-8')}{\uzs_{ii',99'}}$\end{center}
\begin{center}$ D_{10} =
\Pi_{i=1}^8 \uzs_{ii',10\,10'},\quad D_{11} = \Pi_{i=1}^9
\underset{(10-10')}{\uzs_{ii',11\,11'}},\quad D_{12} =
\Pi_{i=1}^{10} \uzs_{ii',12\,12'},\quad D_{13} = \Pi_{i=1}^{11}
\underset{(12-12')}{\uzs_{ii',13\,13'}}$\end{center}
\begin{center}$D_{14} =
\Pi_{i=1}^{10} \uzs_{ii',14\,14'}Z^2_{13\,13',14\,14'},\quad
D_{15} = \Pi_{i=1}^{13}
\underset{(14-14')}{\uzs_{ii',15\,15'}},\quad D_{16} =
\Pi_{i=1}^{14} \uzs_{ii',16\,16'}$\end{center}
\begin{center}$D_{17} = \Pi_{i=1}^{12}
\underset{(16-16')}{\uzs_{ii',17\,17'}}\cdot
\underset{(16-16')}{\uzs_{12\,12',17\,17'}}\uzs_{15\,15',17\,17'}$
\end{center}
and $$H_1 = Z_{1,1'},\quad H_i = Z^{(3)}_{a',bb'}\T
\widetilde{Z}_{aa'}
 \,\,\emph{for}\, i=4,8,9,11,13,15,16,\, H_i = Z^{(3)}_{aa',b}\T \widetilde{Z}_{bb'}
 \,\,\emph{for}\, i=3,5,7,10,14,$$ $\quad H_i = \varphi(a,b,c), \emph{where}\,\, i=2,6,12$ and $v_i$ is the
 intersection of the lines of $L_a, L_b, L_c$, and $L_b$ is regenerated
 first.\\
\end{thm}

Let $G_2 \doteq \p(\C^2 - B_2)$ be the fundamental group of the
complement of the branch curve.
\begin{prs} \label{prsFundGrpP1_C2}
$G_2$ is isomorphic to a quotient of $\hat{A}(T_2)$, where $T_2$
is the following graph:
%

\begin{center}
\epsfig{file=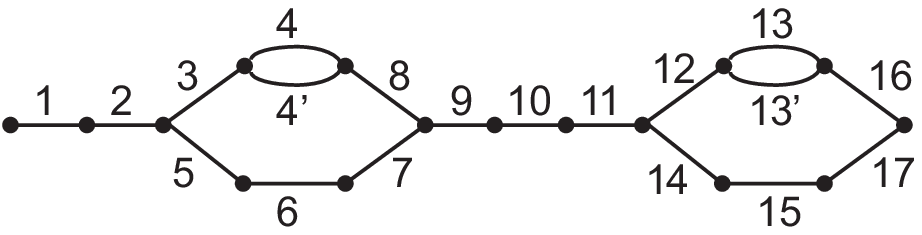}\\ \small{\emph{Figure 22} }
\end{center}

\end{prs}

\begin{proof}
The existence of the relations (i)-(v) as in Definition
\ref{def1} is induced from the braid monodormy factorization of
$B_2$, using the Van-Kampen theorem, as in Propositions
\ref{prs1}, \ref{prs2}, \ref{prs3}.

\end{proof}

\begin{notation}
We introduce the following notations:
\begin{enumerate}
\item [(i)]\emph{Let $T$ be a connected planar graph, with no
repeated edges, and the valence of each vertex is $\leq 3$. We denote these
requirements by $\otimes$.} \item [(ii)] \emph{For a graph $T =
(E,V),\, v \in V$, denote by $E_{T,v} = E_v$ the set of all the
edges in $T$  one of whose ends is $v$.} \item [(iii)] $E^0_v =
E \setminus E_v.$ \item [(iv)] \emph{Let T
be a graph satisfying $\otimes$. Denote by $R(E_v)$ the following expression, induced from the edges in $E_v$:\\
(A) $ E_v = \{u_1,u_2\}$, then $R(E_v) = u_1u_2u_1u_2^{-1}u_1^{-1}u_2^{-1},$ where: \quad \epsfig{file=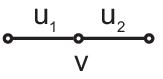}.\\
(B) $ E_v = \{u_1,u_2,u_3\}$, then $R(E_v) =
u_1u_2u_3u_2^{-1}u_1^{-1}u_2u_3^{-1}u_2^{-1}$, where: \quad
\epsfig{file=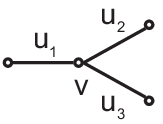}.\\}
\end{enumerate}
\end{notation}

\begin{defi}\emph{Let $T_1 = (V_1,E_1), T_2 =
(V_2,E_2)$ be two graphs satisfying $\otimes$. Assume there exist
two vertices $v_1 \in V_1, v_2 \in V_2$ such that the degree
$d(v_1) = i < 3$ and $d(v_2) \leq 3-i$. We create a new graph $T_1
\bigcup_{v_1}^{v_2} T_2$ by identifying the vertices $v_1$ and
$v_2$. Note that  $T_1 \bigcup_{v_1}^{v_2} T_2$ also satisfies $\otimes$. Let $v$ be the
identified vertex $v_1 = v_2$ in $T_1 \bigcup_{v_1}^{v_2} T_2$.
For example, see the following figure:
\begin{center}
\epsfig{file=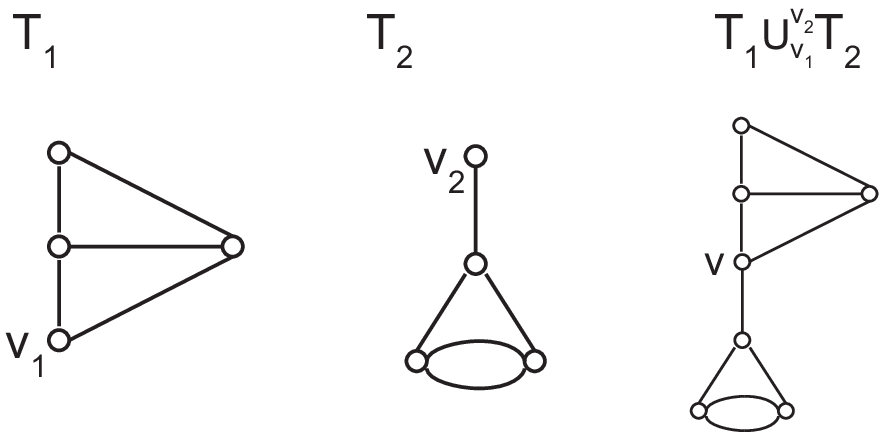} \\
\end{center}}
\end{defi}

\begin{prs} \label{prs5}

$\\\hat A(T_1 \bigcup_{v_1}^{v_2} T_2) = \Big\{ \hat A(T_1) \ast \hat A(T_2)
\Big| \begin{array}{ll} [u_1,u_2] = e,\, \quad u_1 \in E^0_{v_1},
u_2 \in E_2 \,\,\emph{or}\,\, u_2 \in E^0_{v_2}, u_1 \in E_1
\\ R(E_v)= e
\end{array} \Big\}.$

\end{prs}

\begin{proof} We first note that the degree of $v_1$
is less than 3, so the only possible cases are:\\
(a) \quad \epsfig{file=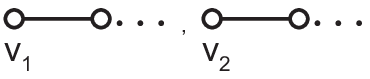} \quad(b) \quad
\epsfig{file=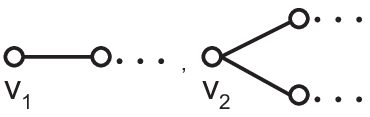} \quad
(c) \quad \epsfig{file=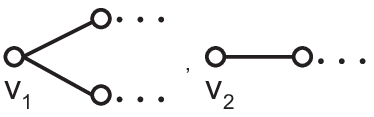}.\\
Cases (b) and (c) are actually the same, so we consider only cases
(a) and (b). Since the edges of $T_1, T_2$ are not changed under
the identification of $v_1$ and $v_2$, it is obvious that the
relations in $\hat A(T_1)$ and $\hat A(T_2)$ are satisfied in
$\hat A(T_1 \bigcup_{v_1}^{v_2} T_2)$. In addition, for an edge
$u_1 \in E_1$ such that $u_1 \not\in E_{v_1},\,u_1$ is disjoint
from any edge $u_2 \in E_2$. Thus, in $\hat A(T_1
\bigcup_{v_1}^{v_2} T_2)$, the generator corresponding to $u_1$
commutes with any generator corresponding to $u_2$. The same is
true for an edge $u_2 \not\in E_{v_2}$ and edges in $E_1$. We only
have to take into account the relation induced from the
identification of $v_1$ and $v_2$. Consider case (a). $E_v$ is a
set of two adjacent edges $u,w$, intersecting at $v$. So in
$\hat A(T_1 \bigcup_{v_1}^{v_2} T_2)$, by Definition~(\ref{def1})(ii), we would have the relation $vwv = wvw$, or
$R(E_v)=e$. We follow the same arguments for case (b).
\end{proof}

\begin{notation}\label{notXg}
\begin{enumerate}
\item [(i)] \emph{Let $T_1 = (V_1,E_1)$ be the graph in
proposition \ref{prs4}, $T_0 = (V_0,E_0)$ , and let $\delta \in
V_1,\,\alpha, \beta \in V_0$ be the following vertices:}
\begin{center}
\epsfig{file=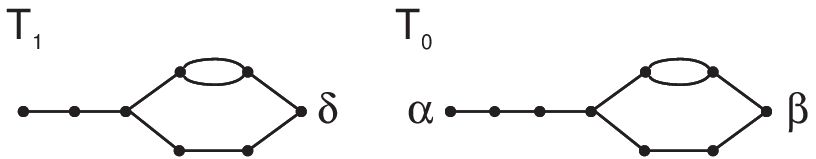}
\end{center}
\item [(ii)] \emph{For $1 < g$, take $g-1$ copies of $T_0$, and
denote by $\alpha_i, \beta_i,\, 1\leq i \leq g$ the corresponding
vertices in each $T_0$. Let
 $T_g \doteq T_1 \bigcup_{\delta}^{\alpha_1}T_0 \bigcup_{\beta_1}^{\alpha_2} ...
  \bigcup_{\beta_{g-1}}^{\alpha_{g}}T_0$.
}\item [(iii)]\emph{We now construct a degenerated model of $\CP^1
\times C_g$, where $C_g$ is a genus $g$ curve. Embed $\CP^1 \times
\CP^1$ by the linear system $(1,3g)$, degenerate it to a union of
$6g$ planes, attach $g$ quadrics to $g$ pairs of non--intersecting
planes and then degenerate the quadrics, as was done in Constructions \ref{cons1} and \ref{cons2}. The resulting
degeneration should be composed from $g$ ``building blocks" as in
figure 17. Explicitly
\begin{center}
\epsfig{file=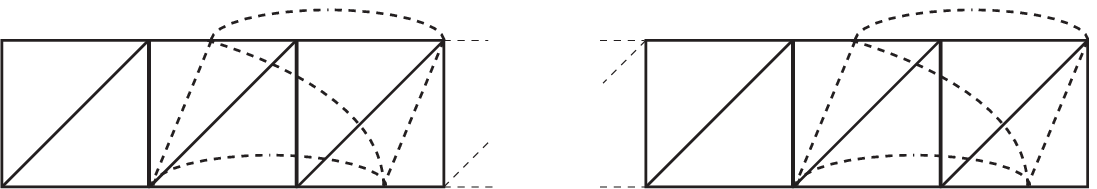}\\ \small{Figure 23 : degeneration of
$\CP^1 \times C_g$ embedded in $\CP^{8g+1}$}
\end{center}
Denote by $S_g$ this surface whose degeneration is as in Figure
24 above. Consider a generic projection $\CP^{8g+1} \to \CP^2$
and its restriction to $S_g$, we denote by $B_g$ the branch curve
and $G_g = \p(\C^2 - B_g)$ the corresponding fundamental group.}
 \end{enumerate}
\end{notation}

 We saw (Proposition \ref{prs4}) that $G_1 \simeq \hat A(T_1)/\rho_1$ and that $G_2$ is a quotient of
$\hat A(T_2)$. Thus, by induction, we have the following
\begin{thm} \label{thm1}

$G_g$ is isomorphic to a quotient of $\hat A(T_g)$.
\end{thm}

 \subsection{The fundamental group of the Galois cover of $\CP^1 \times C_g$}
\label{subSecGalCov} In this subsection we find the fundamental
group of the Galois cover of $\CP^1 \times C_g$, generalizing the
results of \cite{AGTV}, \cite{AG} and using the method outlined in
\cite{L}. We start with reviewing the known facts on the
fundamental group of the Galois cover of a surface.

 Let $S$ be a projective surface of degree $n$. Given a generic
 projection $\pi : S \to \pp^2$, we define the Galois cover as the closure of
the $n$-fold fibered product $S_{Gal} = \overline{S \times_{\pi}
... \times_{\pi} S - \Delta}$ where $\Delta$ is the generalized
diagonal. We denote by $S^{aff}_{Gal}$ the affine part of
$S_{Gal}$.

 Let $B$ be the branch curve of $\pi : S \to \pp^2$. It is known that we have
 the following exact sequences (see e.g., \cite{MoTeSim}):

 $$
0 \rightarrow \pi_1(S^{aff}_{Gal}) \rightarrow \pi_1(\C^2 -
B)/\langle \Gamma^2 = 1 \rangle \rightarrow Sym_n \rightarrow 0,
$$
\begin{equation} \label{eqnGalSeq}
 0 \rightarrow \pi_1(S_{Gal}) \rightarrow \pi_1(\CP^2 - B)/\langle \Gamma^2 = 1 \rangle \rightarrow Sym_n \rightarrow 0.
 \end{equation}

Let $\delta = \prod \G_i$ the product of all the standard
topological generators of  $\pi_1(\C^2 - B)$. Recall that
$\pi_1(\CP^2 - B) = \pi_1(\C^2 - B)/\langle \delta = 1 \rangle$.
Then, by \cite[Proposition 5.10]{L} and Theorem \ref{thm1}, we
see that for the surface $S_g$ (see notation
\ref{notXg}(iii))

  \begin{equation} \label{eqnGalGroup}
  H_1((S_g)_{Gal}) \simeq \mathbb{Z}^{2g(8g-1)}.
  \end{equation}

We consider $\delta$ as an element in $\pi_1(S^{aff}_{Gal})$.
Denote $Z = \langle \delta \rangle \cap \mathbb{Z}$ (see \cite[Thm
4.5]{L}). Then we have the following exact sequence (see
\cite[Proposition 5.10]{L}):

\begin{equation} \label{eqnExactSeqXg}
0 \to \mathbb{Z} / Z \to \pi_1((S_g)_{Gal}) \to
\mathbb{Z}^{2g(8g-1)} \to 0.
\end{equation}

In order to compute $\pi_1((S_g)_{Gal})$ we need the following
definition.

\begin{defi}
\emph{The \textsl{generalized Coxeter group} $\hat{C}(T)$ associated to a
graph $T$ is defined as
$$
\hat C(T) = \hat{A}(T) / \langle  \Gamma^2 = 1 \rangle
$$
where $\G$ goes over all the generators of $\hat{A}(T)$.}
\end{defi}

Note that $\pi_1((S_g)_{Gal})$ is a subgroup of a quotient of
$\hat C(T_g)$, by Theorem \ref{thm1} and the short exact sequence (\ref{eqnGalSeq}).

\begin{thm}
$$
\pi_1((S_g)_{Gal}) \simeq \mathbb{Z}^{2g(8g-1)}.
$$
\end{thm}

\begin{proof}
We will prove the theorem only for $g=1$, where for the $g>1$ the proof is similar.
Let us consider the following group
$$
H = G_1/\langle \G^2 = 1, \G_4 = \G_{4'}\rangle =  \hat C(T_1)/ \langle \G_4 = \G_{4'} \rangle.
$$

Examining the relations in $G_1$, we see that the relation $\rho_1$ becomes trivial under the new added
relations (see Notation \ref{NotPho} and Proposition \ref{prs4}). Therefore, the group $H$ is in fact
isomorphic to the following Coxeter group
$ H \simeq C_Y(T)$ (see \cite{RTU} for the definition of the Coxeter group $C_Y(T)$), where $T$ is as in the figure below.

\begin{center}
\epsfig{file=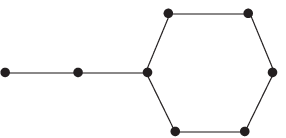}\\ \small{Figure 24 : The graph $T$
associated to the Coxeter group $H$}
\end{center}
Therefore, by \cite[Theorem 6.1]{RTU}, $H \simeq Sym_8 \ltimes
A_{1,8} = Sym_8 \ltimes \mathbb{Z}^7$ (see \cite{RTU} for the notation of $A_{t,n}$). As this group is infinite,
its associated Coxeter element $\prod \G_i$ has infinite order
(see e.g. \cite[pg. 175]{H}). Thus its order is infinite also in
the group $G_1/\langle \G^2 = 1 \rangle$ and thus in any subgroup
of it, for example in $\pi_1((S_g)^{aff}_{Gal})$. Therefore, the
order of $\delta \in \pi_1(S^{aff}_{Gal})$ is infinite, and $Z =
\langle \delta \rangle \cap \mathbb{Z} = \mathbb{Z}$. Considering
the exact sequence in (\ref{eqnExactSeqXg}), we are done.
\end{proof}

\end{document}